\documentclass[12pt,a4paper,reqno]{amsart}
\usepackage[T1]{fontenc}
\usepackage[utf8]{inputenc} 
\usepackage[english]{babel}
\usepackage{lmodern} 
\usepackage{amsmath, amssymb, amsfonts, amsthm}
\usepackage{graphicx}
\usepackage{multicol} 
\usepackage{geometry} 
\usepackage{hyperref} 
\usepackage{xcolor} 
\usepackage{fancyhdr} 

\usepackage{float}
\usepackage{cite}
\usepackage{enumerate}     

\hypersetup{
  colorlinks = true,
  linkcolor = blue,
  citecolor = magenta,
  urlcolor  = teal,
  pdftitle  = {Research Article},
}

\newtheorem{proof.}{Proof}
\newtheorem{theorem}{Theorem}
\newtheorem{coro}{Corollaire}[section]
\newtheorem{lem}{Lemma}
\newtheorem{step}{Step}[section]
\newtheorem{def.}{Definition}[section]
\newtheorem{remark}{Remark}[section]

\title{Sums of three Fibonacci numbers as concatenations of three repdigits in base $b$}

\date{\today} 

\author{Passimzouwé Dagou}
\address{Department of Mathematics, University of Kara, Togo}
\email{\url{passimzouwedagou@gmail.com}}

\author{Pagdame Tiebekabe}
\address{Department of Mathematics, University of Kara, Togo}
\email{\url{tpagdame.math@gmail.com}}

\author{Kouèssi Norbert Adédji}
\address{ Institute of Mathematics and Physics, Department of Mathematics, University of Abomey-Calavi, 072 BP. 50
	Cotonou, Benin Republic.}
\email{\url{adedjnorb1988@gmail.com}}

\author{Kokou Tchariè}
\address{Department of Mathematics, University of Kara, Togo}
\email{\url{tkokou09@gmail.com}}

\begin{document}
\pagestyle{fancy}
\maketitle
\textbf{Abstract}. 
In this paper, we investigate sums of three Fibonacci numbers that can be expressed as concatenations of three repdigits in base $b$, where $b\ge 2$ is an integer. We prove that for bases $2\le b\le 10$, only finitely many such sums exist, and we determine all of them explicitly. Among these solutions, the largest occurs for $b=4$ and is given by
$$
F_{42}+F_{29}+F_{20}=268435290=\overline{33333333331122}_4.
$$


\section{Introduction}
\label{sect1}

The interaction between recurrence sequences and digital representations of integers has attracted considerable attention in recent years. In particular, problems concerning the occurrence of special digital patterns within classical sequences such as the Fibonacci sequence have proved to be a rich source of Diophantine questions combining number theory, recurrence relations, and combinatorial properties of numeration systems. Among these patterns, \emph{repdigits}—integers whose digits are all identical in a given base—and their \emph{concatenations} play a central role.

In a seminal paper published in 2019, Alahmadi, Altassan, Luca, and Shoaib (see \cite{AALS:2019}) initiated a systematic study of Fibonacci numbers that can be expressed as concatenations of two repdigits. Their work opened a new line of research, leading to several subsequent contributions that extended the notion of concatenation and repdigits to other linear recurrence sequences. Notably, analogous problems have been investigated for sequences such as Pell, Pell--Lucas, Padovan, Perrin, and balancing numbers, highlighting the broad relevance and flexibility of this framework (see \cite{a}, \cite{a1}, \cite{a3},  \cite{Ddamulira}, \cite{Qu-Zeng:2020}, \cite{Rayaguru-Panda:2020}, \cite{a2}).

Motivated by these developments, and after a careful review of the existing literature, we address a natural and yet unexplored question:

 \emph{can sums of Fibonacci numbers admit representations as concatenations of repdigits?} 
 
 More precisely, in this paper we study the sum of three Fibonacci numbers that can be written as the concatenation of three repdigits in base $b$. The main objective of this work is to characterize such representations, establish finiteness results where applicable, and determine all possible solutions under suitable constraints. Our results extend and complement previous studies on digital properties of Fibonacci numbers, and contribute further insight into the interplay between recurrence sequences and digital concatenation phenomena.

Let $\{F_n\}_{n\ge 0}$ be  the Fibonacci sequence given by $F_{n+2}=F_{n+1} + F_n,$ with initial values $F_0=0$ and $F_1=1.$  If 
$$
(\alpha, \beta) =\left(\dfrac{1+\sqrt{5}}{2},\dfrac{1-\sqrt{5}}{2} \right)
$$
is the pair of roots of the characteristic equation $x^2 -x-1=0$ of the Fibonacci numbers, then the Binet's formula for its general term is
\begin{align}
\label{eq2}
F_n=\dfrac{\alpha^n-\beta^n}{\alpha-\beta}\quad n\ge 0.
\end{align}
It can be seen that $1<\alpha <2,\;-1<\beta <0$ and $\alpha\beta=-1.$ The following relations between $n$-th Fibonacci number $F_n$  and $\alpha$ is well known
\begin{align}
\label{eq3}
\alpha^{n-2} <F_n<\alpha^{n-1}\quad n\ge 0.
\end{align}
We recall that a positive integer $R$ is called a base $b$-repdigit if its all digits are the same in base $b.$  That is, $R$ is of the form
$$
R=\dfrac{d(b^m-1)}{b-1}=\overline{\underbrace{d\ldots d}_{m~times}}~_{(b)},
$$
for some positive integers $d, m$ with $1\le d\le b-1$ and $m \ge 1.$ When $b=10,$ we omit the base and we simply say that $R$ is a repdigit.  Therefore we
deal with the following  Diophantine equation
\begin{align}
\label{eq1}
F_{n_1}+F_{n_2}+ F_{n_3}=\overline{\underbrace{d_1\ldots d_1}_
	{\ell_1~times}\underbrace{d_2\ldots d_2}_
	{\ell_2~times}\underbrace{d_3\ldots d_3}_
	{\ell_3~times}}~_{(b)}
\end{align}
where $d_1, \ell_1, \ell_2, \ell_3\ge 1$ and $d_1, d_2, d_3\in \{0, 1,\ldots, b-1\}.$ 
Therefore the main results of this paper are stated in the following theorems.
\begin{theorem}
\label{thrm1}
Let $(n_1, n_2, n_3, \ell_1, \ell_2, \ell_3, d_1, d_2, d_3)$ be a solution to equation \eqref{eq1} satisfying
\[
n_1 \ge n_2 \ge n_3 \ge 0, \quad 
\ell_3 \ge \ell_2 \ge \ell_1 \ge 1, \quad 
d_1, d_2, d_3 \in \{0, 1, 2, \dots, b-1\}, \quad d_1>0.
\]
Then we have the explicit upper bound
\[
n_1 < 3.09 \times 10^{82} \, \log^{11} b.
\]
\end{theorem}
\begin{theorem}
\label{thrm2}
For each integer base $b$ with $2 \le b \le 10$, there exist a finite number of sums of three Fibonacci numbers that can be written as concatenations of three repdigits in base $b$. In total, there are $2665$ such distinct sums across all bases. 

Moreover, the following table presents, for each base $b$, 
\begin{itemize}
    \item the number of distinct sums, 
    \item the largest sum $F_{n_1}+F_{n_2}+F_{n_3}$, 
    \item and its representation in base $b$.
\end{itemize}

\begin{center}
\begin{table}[H]
\caption{Largest sums of three Fibonacci numbers represented as concatenations of three repdigits in bases $2$ to $10$}
\renewcommand{\arraystretch}{1.3}
\scriptsize{
\begin{tabular}{|c|c|c|c|}
\hline
$b$ & $2$ & $3$ & $4$ \\ \hline
Number of distinct sums & $113$ & $138$ & $217$ \\ \hline
Largest sum $F_{n_1}+F_{n_2}+F_{n_3}$ & $F_{23}+F_{8}+F_{2}=28679$ & $F_{27}+F_{14}+F_{9}=196829$ & $F_{42}+F_{29}+F_{20}=268435290$ \\ \hline
Representation & $111000000000111$ & $100222222222$ & $33333333331122$ \\ \hline
\end{tabular}
}

\vspace{2mm}

\scriptsize{
\begin{tabular}{|c|c|c|c|}
\hline
$b$ & $5$ & $6$ & $7$ \\ \hline
Number of distinct sums & $250$ & $334$ & $348$ \\ \hline
Largest sum $F_{n_1}+F_{n_2}+F_{n_3}$ & $F_{27}+F_{14}+F_{10}=196850$ & $F_{42}+F_{29}+F_{20}=24186563$ & $F_{34}+F_{24}+F_{17}=5750852$ \\ \hline
Representation & $22244400$ & $2222222455$ & $66611222$ \\ \hline
\end{tabular}
}

\vspace{2mm}

\scriptsize{
\begin{tabular}{|c|c|c|c|}
\hline
$b$ & $8$ & $9$ & $10$ \\ \hline
Number of distinct sums & $387$ & $413$ & $465$ \\ \hline
Largest sum $F_{n_1}+F_{n_2}+F_{n_3}$ & $F_{29}+F_{25}+F_{23}=617911$ & $F_{33}+F_{21}+F_{9}=3535558$ & $F_{35}+F_{28}+F_{26}=3535558$ \\ \hline
Representation & $2266667$ & $6577777$ & $9666669$ \\ \hline
\end{tabular}
}
\end{table}
\end{center}
\end{theorem}

As immediate consequences of the main theorems, we obtain the following corollaries.

\begin{coro}
In equation \eqref{eq1}, setting $n_3 = 0$ yields all distinct sums of two Fibonacci numbers that can be expressed as concatenations of three repdigits in base $b$. In particular we have the following results listened in the table below.

\begin{center}
	\begin{table}[H]
		\caption{Largest sums of two Fibonacci numbers represented as concatenations of three repdigits in bases $2$ to $10$}
		\renewcommand{\arraystretch}{1.3}
		\scriptsize{
			\begin{tabular}{|c|c|c|c|}
				\hline
				$b$ & $2$ & $3$ & $4$ \\ \hline
				Number of distinct sums & $34$ & $42$ & $49$ \\ \hline
				Largest sum $F_{n_1}+F_{n_2}$ & $F_{21}+F_{17}=12543$ & $F_{23}+F_{16}=29644$ & $F_{21}+F_{17}=12543$ \\ \hline
				Representation & $11000011111111$ & $1111122221$ & $3003333$ \\ \hline
			\end{tabular}
		}
		
		\vspace{2mm}
		
		\scriptsize{
			\begin{tabular}{|c|c|c|c|}
				\hline
				$b$ & $5$ & $6$ & $7$ \\ \hline
				Number of distinct sums & $57$ & $71$ & $65$ \\ \hline
				Largest sum $F_{n_1}+F_{n_2}$ & $F_{27}+F_{12}=196562$ & $F_{37}+F_{23}=24186474$ & $F_{21}+F_{13}=11179$ \\ \hline
				Representation & $22242222$ & $2222222230$ & $44410$ \\ \hline
			\end{tabular}
		}
		
		\vspace{2mm}
		
		\scriptsize{
			\begin{tabular}{|c|c|c|c|}
				\hline
				$b$ & $8$ & $9$ & $10$ \\ \hline
				Number of distinct sums & $81$ & $78$ & $80$ \\ \hline
				Largest sum $F_{n_1}+F_{n_2}$ & $F_{27}+F_{12}=196562$ & $F_{21}+F_{19}=15127$ & $F_{28}+F_{22}=335522$ \\ \hline
				Representation & $577722$ & $22667$ & $335522$ \\ \hline
			\end{tabular}
		}
	\end{table}
\end{center}
\end{coro}

\begin{coro}
In equation \eqref{eq1}, setting $n_2 = n_3 = 0$ yields all Fibonacci numbers that can be expressed as concatenations of three repdigits in base $b$. In particular we get the following results listened in the table bellow.

\begin{center}
	\begin{table}[H]
		\caption{Largest Fibonacci numbers represented as concatenations of three repdigits in bases $2$ to $10$}
		\renewcommand{\arraystretch}{1.3}
		\scriptsize{
			\begin{tabular}{|c|c|c|c|}
				\hline
				$b$ & $2$ & $3$ & $4$ \\ \hline
				Number of distinct sums & $4$ & $6$ & $7$ \\ \hline
				Largest $F_{n}$ & $F_{10}=55$ & $F_{14}=377$ & $F_{19}=4181$ \\ \hline
				Representation & $110111$ & $111222$ & $1001111$ \\ \hline
			\end{tabular}
		}
		
		\vspace{2mm}
		
		\scriptsize{
			\begin{tabular}{|c|c|c|c|}
				\hline
				$b$ & $5$ & $6$ & $7$ \\ \hline
				Number of distinct sums & $5$ & $6$ & $5$ \\ \hline
				Largest  $F_{n}$ & $F_{15}=610$ & $F_{24}=46386$ & $F_{17}=1597$ \\ \hline
				Representation & $4420$ & $554400$ & $4441$ \\ \hline
			\end{tabular}
		}
		
		\vspace{2mm}
		
		\scriptsize{
			\begin{tabular}{|c|c|c|c|}
				\hline
				$b$ & $8$ & $9$ & $10$ \\ \hline
				Number of distinct sums & $6$ & $6$ & $6$ \\ \hline
				Largest  $F_{n}$ & $F_{16}=987$ & $F_{19}=4181$ & $F_{22}=17711$ \\ \hline
				Representation & $1733$ & $5655$ & $17711$ \\ \hline
			\end{tabular}
		}
	\end{table}
\end{center}
\end{coro}

The proofs of our results rely crucially on lower bounds for linear forms in logarithms, combined with Baker's reduction method. Furthermore, our computations are carried out using the approach developed by Chim and Ziegler in \cite{CV}.

\section{Preliminaries}

\subsection{Linear form in logarithms}
Denote by $\eta_1,\ldots,\eta_s$ algebraic numbers, none of which is $0$ or $1$, and let 
$\log \eta_1,\ldots,\log \eta_s$ be a fixed determination of their logarithms.  
Let $\mathbb{K}=\mathbb{Q}(\eta_1,\ldots,\eta_s)$ and denote by $d_{\mathbb{K}}=[\mathbb{K}:\mathbb{Q}]$ the degree of $\mathbb{K}$ over $\mathbb{Q}$. 

For any $\eta\in \mathbb{K}$, let its minimal polynomial over $\mathbb{Z}$ be
\[
f(X)=a_0 X^\delta + a_1 X^{\delta-1} + \cdots + a_\delta
= a_0 \prod_{j=1}^{\delta} (X - \eta^{(j)}),
\]
where $\eta^{(j)}$, $j=1,\ldots,\delta$, are the roots of $f(X)$.  
The \emph{absolute logarithmic Weil height} of $\eta$ is defined by
\[
h_0(\eta)=\frac{1}{\delta}\left(
\log |a_0| + \sum_{j=1}^{\delta} \log\big(\max\{ |\eta^{(j)}|,1\}\big)
\right).
\]
In particular, if $\eta=p/q$ is a rational number in lowest terms, i.e., $\gcd(p,q)=1$ and $q>0$, then $h_0(\eta)=\log\max(|p|,q)$.  

The \emph{modified height} $h'(\eta)$ is defined by
\[
h'(\eta)=\frac{1}{d_{\mathbb{K}}}\max\{ h(\eta), |\log \eta|, 1 \},
\]
where $h(\eta)=d_{\mathbb{K}}\, h_0(\eta)$ is the standard logarithmic Weil height of $\eta$.

Consider the linear form
\[
L(z_1,\ldots,z_s)= b_1 z_1 + \cdots + b_s z_s,
\]
where $b_1,\ldots,b_s$ are integers, not all zero. Its modified height is
\[
h'(L)=\frac{1}{d_{\mathbb{K}}}\max\{ h(L), 1 \},
\]
where 
\[
h(L)= d_{\mathbb{K}} \log\left( \max_{1\le j\le s} \left\{ \frac{|b_j|}{b} \right\} \right)
\]
is the logarithmic Weil height of $L$, with $b = \gcd(b_1,\ldots,b_s)$.  

If we define $B = \max\{|b_1|,\ldots,|b_s|, e\}$, then we have the simple bound
\[
h'(L) \le \log B.
\]

With these notations, we are now able to state the following fundamental result due to Baker and Wüstholz \cite{Baker-Wustholz:1993}.

\begin{theorem}
	\label{thrm3}
If $\Gamma = L(\log \eta_1,\ldots,\log \eta_s) \ne 0$, then
\[
\log|\Gamma| \ge - C(s,d_{\mathbb{K}})\, h'(\eta_1)\cdots h'(\eta_s)\, h'(L),
\]
where
\[
C(s,d_{\mathbb{K}}) = 18(s+1)! \, s^{s+1} (32d_{\mathbb{K}})^{s+2} \log(2sd_{\mathbb{K}}).
\]

With $|\Lambda|\le \tfrac{1}{2}$, where
\[
\Lambda = e^{\Gamma} - 1 = \eta_1^{b_1}\cdots \eta_s^{b_s} - 1,
\]
we have $\tfrac12|\Gamma| \le |\Lambda| \le 2|\Gamma|$ so that
\[
\log\left|\eta_1^{b_1}\cdots \eta_s^{b_s} - 1\right|
\ge \log|\Gamma| - \log 2.
\]
\end{theorem} 
We apply Theorem \ref{thrm3} mainly in the situation where $\mathbb{K}=\mathbb{Q}(\sqrt{5})$, $s=3$ and $d_{\mathbb{K}}=2$.  In this case we obtain
\[
C(3,2)=18\cdot 4! \cdot 3^4 \cdot 64^5 \log 12 < 9.34\cdot 10^{13}.
\]
This value will be used throughout the paper without further mention. We recall below some well-known properties of the absolute logarithmic height.
\[
h_0(\eta+\gamma) \le h_0(\eta)+h_0(\gamma)+\log 2,
\]
\[
h_0(\eta\gamma^{\pm 1}) \le h_0(\eta)+h_0(\gamma),
\]
\[
h_0(\eta^\ell) = |\ell|\,h_0(\eta),
\]
where $\eta,\gamma$ are algebraic numbers and $\ell\in\mathbb{Z}$. We also need the following result from Guzm\'an and Luca\cite{SGL}.
\begin{lem}(Lemma 7 of \cite{SGL}) ~\\ \label{lem2}
	If $l\geq1, H>(4l^2)^l$ and $H>L/(\log L)^l$, then $L<2^lH(\log H)^l$.
\end{lem}
\subsection{Reduction method}
The bounds for the variables derived from Baker’s theory are far too large to allow for effective computational verification. To reduce the bounds we use the reduction method due to   Bravo,  G\'omez  and Luca     \cite{BGF} which is a modified version due to Dujella and Peth\H{o} (see \cite{Dujella-Peto:1998}).
\begin{lem}
	\label{lem3}
	Let $M$ be a positive integer, $p/q$ be a convergent of the continued fraction expansion of irrational number $\tau$ such as $q>6M$, and $A, B, \mu$ be some real numbers with $A>0$ and $B>1.$ Furthermore, let $$ \varepsilon:=||\mu q||-M.||\tau q||. $$
	If $\varepsilon>0$, the there is no solution to the inequality 
	$$  0<|u\tau-v+\mu|<AB^{-w}$$
	in positive integers $u, v$ and $w$ with $$ u\leq M \text{ and } w\geq \frac{\log(Aq/\varepsilon)}{\log B}.$$
	Here, for a real number $X$, $||X||$ denotes the distance from $X$ to the nearest integer, that is, $||X||:=\min_{n\in\mathbb{Z}}|X-n|$. 
\end{lem}
\section{Proof the Theorem~\ref{thrm1}}
\label{sect3}

For the proof of Theorem~\ref{thrm1}, we may assume without loss of generality that 
$n_1> 300$, since the remaining cases can be checked directly by computation. The proof of Theorem~\ref{thrm1} is divided into several steps. We derive initial bounds for the variables using estimates coming from linear forms in logarithms.
\begin{step}

	\label{step3.1} 
We prove the following Lemma.
\begin{lem}
\label{lem4}
All the solution of equation (\ref{eq1}) satisfy $$
\min\left\lbrace \ell_{1}\log b, (n_1-n_2)\log\alpha \right\rbrace<1.5\times10^{14}\log^2b\log(n_1+2).
$$
\end{lem}
\end{step}	
\begin{proof}
As a first step, we analyze the Diophantine equation \eqref{eq1} and transform it into the equivalent form
\begin{align*}
F_{n_{1}}+ F_{n_{2}}+ F_{n_3}&=\overline{
	\underbrace{d_1\,\ldots\,d_1}_{\ell_1\ \text{times}}\;
	\underbrace{d_2\,\ldots\,d_3}_{\ell_{2}\ \text{times}}\;
	\underbrace{d_3\,\ldots\,d_3}_{\ell_{3}\ \text{times}}
}\\
&= \underbrace{d_1\,\ldots\,d_1}_{\ell_1\ \text{times}}.b^{\ell_{2}+\ell_{3}}+\underbrace{d_2\,\ldots\,d_2}_{\ell_2\ \text{times}}.b^{\ell_{3}}+\underbrace{d_3\,\ldots\,d_3}_{\ell_3\ \text{times}}\\
F_{n_{1}}+ F_{n_{2}}+ F_{n_3}&=d_1\left(\frac{b^{\ell_{1}}-1}{b-1} \right) .b^{\ell_{2}+\ell_{3}}+d_2\left(\frac{b^{\ell_{2}}-1}{b-1} \right).b^{\ell_{3}} +d_3\left(\frac{b^{\ell_{3}}-1}{b-1} \right) 
\end{align*}
and therefore we get
\begin{equation}
\label{eq4}
F_{n_{1}}+ F_{n_{2}}+ F_{n_3}=\frac{1}{b-1}\left(d_1b^{\ell_{1}+\ell_{2}+\ell_{3}}-(d_1-d_2)b^{\ell_{2}+\ell_{3}}-(d_2-d_3)b^{\ell_{3}}-d_3 \right).
\end{equation}
We first observe that combining \eqref{eq3} and \eqref{eq1} yields
$$
b^{\ell_{1}+\ell_{2}+\ell_{3}}>F_{n_1}+F_{n_2}+F_{n_3}>F_{n_1}>\alpha^{n_1-2}               
$$ 
and
\begin{equation}
\label{eq5}
3\alpha^{n_1-1}>3F_{n_1}>F_{n_1}+F_{n_2}+F_{n_3}\geq b^{\ell_{1}+\ell_{2}+\ell_{3}-1}.
\end{equation}
By taking logarithm on the both sides of (\ref{eq5}), we obtain
$$
(\ell_{1}+\ell_{2}+\ell_{3}-1)\log b<(n_1-1)\log\alpha+\log3
$$
which leads to 
$$
\ell_{1}+\ell_{2}+\ell_{3}<n_1+2.
$$
Using (\ref{eq2}) and (\ref{eq4}), we get
$$
\frac{\alpha^{n_{1}}-\beta^{n_{1}}}{\sqrt{5}}+F_{n_2}+F_{n_3}=\frac{1}{b-1}\left(d_1b^{\ell_{1}+\ell_{2}+\ell_{3}}-(d_1-d_2)b^{\ell_{2}+\ell_{3}}-(d_2-d_3)b^{\ell_{3}}-d_3 \right)
$$
which implies that
\begin{align*}
(b-1)\alpha^{n_1} - \sqrt{5}\, d_1 b^{\ell_1+\ell_2+\ell_3} 
&= -\sqrt{5}\Big((d_1-d_2)b^{\ell_2+\ell_3} + (d_2-d_3)b^{\ell_3} + d_3\Big) \\
&\quad - \sqrt{5}\,(b-1)(F_{n_2}+F_{n_3}) + (b-1)\beta^{n_1}, \\[1.5mm]
\Big|(b-1)\alpha^{n_1} - \sqrt{5}\, d_1 b^{\ell_1+\ell_2+\ell_3}\Big|
&< \sqrt{5}\Big((b-1)b^{\ell_2+\ell_3} + (b-1)b^{\ell_3} + (b-1)\Big) \\
&\quad + \sqrt{5}\,(b-1)(F_{n_2}+F_{n_3}) + (b-1)|\beta^{n_1}|, \\[1mm]
&< \sqrt{5}\Big((b-1)b^{\ell_2+\ell_3} + (b-1)b^{\ell_3} + (b-1)\Big) \\
&\quad + 2\sqrt{5}\,(b-1)\alpha^{n_2-1} + (b-1)\alpha^{-n_1}, \\[1mm]
&< \sqrt{5}\,(b-1)b^{\ell_2+\ell_3}\Big( 1 + \frac{1}{b^{\ell_2}} + \frac{1}{b^{\ell_2+\ell_3}} \Big) \\
&\quad + 2\sqrt{5}\,(b-1)\alpha^{n_2-1} + (b-1)\alpha^{-n_1}, \\[1mm]
&< 1.75 \sqrt{5}\, b^{\ell_2+\ell_3+1} + 2\sqrt{5}\, b \alpha^{n_2-1} + b \alpha^{-n_1}, \\[1mm]
&< 1.75 \sqrt{5}\, b^{\ell_2+\ell_3+1} + 4.48\, b \alpha^{n_2-1}.
\end{align*}

Next,
\begin{equation}
\label{eq6}
\left|(b-1)\alpha^{n_{1}}-\sqrt{5}d_{1}b^{\ell_{1}+\ell_{2}+\ell_{3}}\right|< 6.7b\max\{b^{\ell_{2}+\ell_{3}},\alpha^{n_2}\}
\end{equation}
where we used the fact that $n_1 > 300$. Dividing both sides of (\ref{eq6}) by $\sqrt{5}d_1b^{\ell_{1}+\ell_{2}+\ell_{3}}$, we obtain
\begin{align}
\left|\left( \frac{b-1}{\sqrt{5}d_1}\right) .\alpha^{n_1}.b^{-(\ell_{1}+\ell_{2}+\ell_{3})} -1\right| &<
\max\left\lbrace \frac{6.7b^{\ell_{2}+\ell_{3}+1}}{\sqrt{5}b^{\ell_{1}+\ell_{2}+\ell_{3}}}, \frac{6.7b\alpha^{n_2}}{\sqrt{5}b^{\ell_{1}+\ell_{2}+\ell_{3}}}\right\rbrace  \notag\\
&<\max\left\lbrace \frac{6.7b^{-\ell_{1}+1}}{\sqrt{5}}, \frac{6.7b\alpha^{n_2}}{\sqrt{5}\alpha^{n_1-2}}\right\rbrace \notag\\
&<7.85b\max\left\lbrace b^{-\ell_{1}},\alpha^{-(n_1-n_2)}\right\rbrace. \label{eq7}
\end{align}
We now define the linear form
$$
\Gamma_1=\log\left( \frac{b-1}{d_1\sqrt{5}}\right) +n_1\log\alpha-(\ell_{1}+\ell_{2}+\ell_{3})\log b
$$
and put 
\begin{equation}
\label{eq8}
\Lambda_1:= \mathrm{e}^{\Gamma_1}-1= \left( \frac{b-1}{\sqrt{5}d_1}\right) .\alpha^{n_1}.b^{-(\ell_{1}+\ell_{2}+\ell_{3})} -1.
\end{equation}
We assume that $\left|\Lambda_1 \right|\le 0.5 $ and use the theorem of Baker and Wüstholz (Theorem \ref{thrm3}) with data 
$$
\eta_{1}:= \frac{b-1}{d_1\sqrt{5}}, \eta_{2}:= \alpha, \eta_{3}:=b, b_1=1, b_2:= n_1, b_3:=-(\ell_{1}+\ell_{2}+\ell_{3}), s:=3.
$$
Since  $\ell_{1}+\ell_{2}+\ell_{3}<n_1+2$, we can take $B=n_1+2$. Observe that $\mathbb{K}=\mathbb{Q}\left(\eta_{1}, \eta_{2}, \eta_{3} \right)=\mathbb{Q}(\alpha)$, so $d_{\mathbb{K}}:=[\mathbb{K}:\mathbb{Q}]=2$. We proceed to apply Theorem \ref{thrm3} on \eqref{eq8}. We have first to observe that $\Lambda_1\neq0$.  Indeed, if it were equal to zero, it would follow that $\alpha^{n_1}=\frac{\sqrt{5}d_1}{b-1}b^{\ell_{1}+\ell_{2}+\ell_{3}}$. That is, $\alpha^{2n_1}\in\mathbb{Q}$, which is false for any $n_1>0$.
 Using the properties of the logarithmic height, we have
\begin{align*}
h_0(\eta_{1})=h_0\left(\frac{b-1}{\sqrt{5}d_1} \right) &\leq h_0\left(\frac{b-1}{d_1} \right)+h_0\left( \sqrt{5}\right) \\
& \leq h_0(b-1)+h_0(d_1)+\frac{\log5}{2}\\
& \leq 2\log b+\frac{\log5}{2} \leq 3.2\log b \text{ since } b\geq2.
\end{align*}
Then, we have $h'(\eta_{1})<3.2\log b.$ Similarly, we have $h'\left( \eta_{2}\right)=h(\alpha)=\frac{1}{2},\;  h\left( \eta_{3}\right)=h(b)=\log b$. Theorem \ref{thrm3} tells us that
\begin{align*}
\log\left| \Lambda_1\right|&> -C(3,2)\left(3.2\log b\right) \left(\frac{1}{2} \right) \left( \log b\right) \left(\log(n_1+2) \right)-\log2.
\end{align*}
Comparing this last inequality with (\ref{eq7}) leads to 
$$
\min\left\lbrace \ell_{1}\log b, (n_1-n_2)\log\alpha \right\rbrace<1.5\times10^{14}\log^2b\log(n_1+2).
$$
Note that in the case that $\left|\Lambda_1 \right|>0.5 $, inequality (\ref{eq7}) is possible only if either $\ell_{1}\le 4$ and $n_1-n_2\le \left\lfloor \frac{\log(15.7b)}{\log\alpha} \right\rfloor$, which are covered by the bound provided by Lemma \ref{lem4}
\end{proof}

We now distinguish between the following two cases:

\noindent
\textbf{Case 1:} 
\[
\min\left\lbrace \ell_{1}\log b, (n_1-n_2)\log\alpha \right\rbrace = (n_1-n_2)\log\alpha,
\]

\noindent
\textbf{Case 2:} 
\[
\min\left\lbrace \ell_{1}\log b, (n_1-n_2)\log\alpha \right\rbrace = \ell_{1}\log b.
\]

\noindent
We will treat these two cases separately in the following steps.
\begin{step}
			\label{step3.2} 
We consider Case 1 and show that under the assumption that 
$$
(n_1-n_2)\log\alpha<1.5\times10^{14}\log^2b\log(n_1+2)
$$ 
we obtain that 
$$
 \min\left\lbrace\ell_{1}\log b, (n_1-n_3)\log\alpha \right\rbrace <7.48\times10^{27}\log^3b\log^2(n_1+2).
$$
\end{step}
Since we consider Case~1, we assume that 
$$
\min\left\lbrace \ell_{1}\log b, (n_1-n_2)\log\alpha \right\rbrace=(n_1-n_2)\log\alpha<1.5\times10^{14}\log^2b\log(n_1+2).
$$
Using \eqref{eq2} and \eqref{eq4}, we get
$$
\frac{\alpha^{n_{1}}-\beta^{n_{1}}}{\sqrt{5}}+\frac{\alpha^{n_{2}}-\beta^{n_{2}}}{\sqrt{5}}+F_{n_3}=\frac{1}{b-1}\left(d_1b^{\ell_{1}+\ell_{2}+\ell_{3}}-(d_1-d_2)b^{\ell_{2}+\ell_{3}}-(d_2-d_3)b^{\ell_{3}}-d_3 \right)
$$
which implies that
\begin{align*}
(b-1)\alpha^{n_{1}}+(b-1)\alpha^{n_{2}}-\sqrt{5}d_{1}b^{\ell_{1}+\ell_{2}+\ell_{3}}&= -\sqrt{5}\left((d_1-d_2)b^{\ell_{2}+\ell_{3}}+(d_2-d_3)b^{\ell_{3}}+d_3\right)\\
&-\sqrt{5}(b-1)F_{n_3} +(b-1)\left( \beta^{n_1}+\beta^{n_2}\right). 
\end{align*}
Thus
\begin{align*}
\left|(b-1)\alpha^{n_{1}}\left(1+\alpha^{n_2-n_1} \right) -\sqrt{5}d_{1}b^{\ell_{1}+\ell_{2}+\ell_{3}}\right|&< \sqrt{5}\left((b-1)b^{\ell_{2}+\ell_{3}}+(b-1)b^{\ell_{3}}+(b-1)\right)\\
&+\sqrt{5}(b-1)\alpha^{n_3} +2(b-1)\alpha^{-n_2}  \\
&< 1.75\sqrt{5}b^{\ell_{2}+\ell_{3}+1}
+\sqrt{5}b\alpha^{n_3-1}+2b\alpha^{-n_2}
\end{align*}
and therefore
\begin{equation}\label{eq9}
\left|(b-1)\alpha^{n_{1}}\left(1+\alpha^{n_2-n_1} \right) -\sqrt{5}d_{1}b^{\ell_{1}+\ell_{2}+\ell_{3}}\right|< 5.3b\max\{b^{\ell_{2}+\ell_{3}},\alpha^{n_3}\}.
\end{equation}
Dividing both sides of \eqref{eq9} by $\sqrt{5}d_1b^{\ell_{1}+\ell_{2}+\ell_{3}}$, we obtain
\begin{align}
\left|\left( \frac{(b-1)\left( 1+\alpha^{n_2-n_1}\right) }{\sqrt{5}d_1}\right) .\alpha^{n_1}.b^{-(\ell_{1}+\ell_{2}+\ell_{3})} -1\right| &<
\max\left\lbrace \frac{5.3b^{\ell_{2}+\ell_{3}+1}}{\sqrt{5}b^{\ell_{1}+\ell_{2}+\ell_{3}}}, \frac{5.3b\alpha^{n_3}}{\sqrt{5}b^{\ell_{1}+\ell_{2}+\ell_{3}}}\right\rbrace  \notag\\
&<\max\left\lbrace \frac{5.3b^{-\ell_{1}+1}}{\sqrt{5}}, \frac{5.3b\alpha^{n_3}}{\sqrt{5}\alpha^{n_1-2}}\right\rbrace \notag\\
&<9.93b\max\left\lbrace b^{-\ell_{1}},\alpha^{-(n_1-n_3)}\right\rbrace. \label{eq10}
\end{align}
Let set the linear form 
$$
\Gamma_2=\log\left( \frac{(b-1)\left( 1+\alpha^{n_2-n_1}\right) }{\sqrt{5}d_1}\right) +n_1\log\alpha-(\ell_{1}+\ell_{2}+\ell_{3})\log b
$$
and put 
\begin{equation}
\label{eq11}
\Lambda_2:= \mathrm{e}^{\Gamma_2}-1= \left( \frac{(b-1)\left( 1+\alpha^{n_2-n_1}\right) }{\sqrt{5}d_1}\right) .\alpha^{n_1}.b^{-(\ell_{1}+\ell_{2}+\ell_{3})} -1.
\end{equation}
We assume that $\left|\Lambda_2 \right|\le 0.5 $ and use the theorem of Baker and Wüstholz (Theorem \ref{thrm3}) with data 
$$
\eta_{1}:= \left( \frac{(b-1)\left( 1+\alpha^{n_2-n_1}\right) }{\sqrt{5}d_1}\right), \eta_{2}:= \alpha, \eta_{3}:=b, b_1=1, b_2:= n_1, b_3:=-(\ell_{1}+\ell_{2}+\ell_{3}), s:=3.
$$
Since  $\ell_{1}+\ell_{2}+\ell_{3}<n_1+2$, we can take $B=n_1+2$. Observe that $\mathbb{K}=\mathbb{Q}\left(\eta_{1}, \eta_{2}, \eta_{3} \right)=\mathbb{Q}(\alpha)$, so $d_{\mathbb{K}}:=[\mathbb{K}:\mathbb{Q}]=2$. We proceed to apply Theorem \ref{thrm3} on \eqref{eq11}. We have first to observe that $\Lambda_2\neq0$.  Indeed, if it were zero, we would then get that $\alpha^{n_1}+\alpha^{n_2}=\frac{\sqrt{5}d_1b^{\ell_{1}+\ell_{2}+\ell_{3}}}{(b-1) }$. Let $\sigma\neq \text{ id }$ be the unique non-trivial $\mathbb{Q}$-automorphism of $\mathbb{Q}(\sqrt{5})$. Apply $\sigma$ to the both sides, and taking the absolute  values, we obtain
$$
1<\left| \frac{\sqrt{5}d_1b^{\ell_{1}+\ell_{2}+\ell_{3}}}{(b-1) }\right|=\left| \beta^{n_1}+\beta^{n_2}\right|=  \alpha^{-n_1}+\alpha^{-n_2}<1  \text{ for any } n_1>300 \text{ and } n_2>0, 
$$
which is false. Therefore, $\Lambda_2\neq0$. Using the properties of the logarithmic height, we have
\begin{align*}
h_0(\eta_{1})&=h_0\left( \frac{(b-1)\left( 1+\alpha^{n_2-n_1}\right) }{\sqrt{5}d_1}\right)\\& \leq h_0\left(b-1\right)+(n_1-n_2)h_0(\alpha)+h_0(d_1)+\log\sqrt{5}+\log2 \\
& \leq 2\log b+(n_1-n_2)\log\alpha+\log\sqrt{5}+\log2\\
&\le 1.5\times10^{14}\log^2b\log(n_1+2)+2\log b+\log\sqrt{5}+\log2 \\
& < 1.6\times10^{14}\log^2b\log(n_1+2). 
\end{align*}
Then we have $h'(\eta_{1})<1.6\times10^{14}\log^2b\log(n_1+2).$ Similarly, we have $h'\left( \eta_{2}\right)=h(\alpha)=\frac{1}{2},\;  h\left( \eta_{3}\right)=h(b)=\log b$. Theorem \ref{thrm3} tells us that
\begin{align*}
\log\left| \Lambda_2\right|&> -C(3,2)\left(1.6\times10^{14}\log^2b\log(n_1+2)\right) \left(\frac{1}{2} \right) \left( \log b\right) \left(\log(n+2) \right)-\log2\\
&>-7.472\times10^{27}\log^3b\log^2(n_1+2).
\end{align*}
By comparing this inequality with~\eqref{eq10}, we obtain
$$
\min\left\lbrace\ell_{1}\log b, (n_1-n_3)\log\alpha \right\rbrace <7.48\times10^{27}\log^3b\log^2(n_1+2).
$$
Note that in the case where $|\Lambda_2|>0.5$, inequality~\eqref{eq10} can only hold if either $\ell_1 \le 5$ or 
\[
n_1 - n_3 \le \left\lfloor \frac{\log(19.86b)}{\log \alpha} \right\rfloor,
\] 
both of which are already covered by the bounds provided in Step~2. At this stage, we consider the following two sub-cases for Case~1:

\noindent
\textbf{Case 1A:} 
\[
\min\left\lbrace \ell_1 \log b, (n_1 - n_3)\log \alpha \right\rbrace = (n_1 - n_3)\log \alpha,
\]

\noindent
\textbf{Case 1B:} 
\[
\min\left\lbrace \ell_1 \log b, (n_1 - n_3)\log \alpha \right\rbrace = \ell_1 \log b.
\]

\noindent
We will treat these sub-cases separately in the following steps.
\begin{step}\label{step3.3} 
We now consider Case~1A and show that, under the assumptions
\[
(n_1 - n_2)\log \alpha < 1.5 \times 10^{14} \, \log^2 b \, \log(n_1 + 2) \quad \text{and} \quad
(n_1 - n_3)\log \alpha < 7.48 \times 10^{27} \, \log^3 b \, \log^2(n_1 + 2),
\]
it follows that
\[
\ell_1 \log b < 1.76 \times 10^{41} \, \log^4 b \, \log^3(n_1 + 2).
\]
\end{step}
Combining~\eqref{eq2} and~\eqref{eq4}, we deduce
$$
\frac{\alpha^{n_{1}}-\beta^{n_{1}}}{\sqrt{5}}+\frac{\alpha^{n_{2}}-\beta^{n_{2}}}{\sqrt{5}}+\frac{\alpha^{n_{3}}-\beta^{n_{3}}}{\sqrt{5}}=\frac{1}{b-1}\left(d_1b^{\ell_{1}+\ell_{2}+\ell_{3}}-(d_1-d_2)b^{\ell_{2}+\ell_{3}}-(d_2-d_3)b^{\ell_{3}}-d_3 \right)
$$
which implies 
\begin{align*}
&(b-1)\alpha^{n_{1}}+(b-1)\alpha^{n_{2}}+(b-1)\alpha^{n_{3}}-\sqrt{5}d_{1}b^{\ell_{1}+\ell_{2}+\ell_{3}}\\&= -\sqrt{5}\left((d_1-d_2)b^{\ell_{2}+\ell_{3}}+(d_2-d_3)b^{\ell_{3}}+d_3\right)\\
&+(b-1)\left( \beta^{n_1}+\beta^{n_2}+\beta^{n_3}\right)  
\end{align*}
and 
\begin{align*}
&\left|(b-1)\alpha^{n_{1}}\left( 1+\alpha^{n_2-n_{1}}+\alpha^{n_3-n_{1}}\right) -\sqrt{5}d_{1}b^{\ell_{1}+\ell_{2}+\ell_{3}}\right|\\&< \sqrt{5}\left((b-1)b^{\ell_{2}+\ell_{3}}+(b-1)b^{\ell_{3}}+(b-1)\right)\\
& +3(b-1)\alpha^{-n_3}  
< 4b^{\ell_{2}+\ell_{3}+1}.
\end{align*}
We then deduce
\begin{equation}
\label{eq12}
\left|(b-1)\alpha^{n_{1}}\left( 1+\alpha^{n_2-n_{1}}+\alpha^{n_3-n_{1}}\right) -\sqrt{5}d_{1}b^{\ell_{1}+\ell_{2}+\ell_{3}}\right|< 4b^{\ell_{2}+\ell_{3}+1}.
\end{equation}
Dividing both sides of \eqref{eq12} by $\sqrt{5}d_1b^{\ell_{1}+\ell_{2}+\ell_{3}}$, we obtain
\begin{align}
\label{eq13}
\left|\left( \frac{(b-1)\left( 1+\alpha^{n_2-n_1}+\alpha^{n_3-n_1}\right) }{\sqrt{5}d_1}\right) .\alpha^{n_1}.b^{-(\ell_{1}+\ell_{2}+\ell_{3})} -1\right| &<1.8b^{-\ell_{1}+1}.
\end{align}
Let set the linear form 
$$
\Gamma_3=\log\left( \frac{(b-1)\left( 1+\alpha^{n_2-n_1}+\alpha^{n_3-n_1}\right) }{\sqrt{5}d_1}\right)+n_1\log\alpha-(\ell_{1}+\ell_{2}+\ell_{3})\log b
$$
and put 
\begin{equation}
\label{eq14}
\Lambda_3:= \mathrm{e}^{\Gamma_3}-1= \left( \frac{(b-1)\left( 1+\alpha^{n_2-n_1}+\alpha^{n_3-n_1}\right) }{\sqrt{5}d_1}\right) .\alpha^{n_1}.b^{-(\ell_{1}+\ell_{2}+\ell_{3})} -1.
\end{equation}
We apply Theorem \ref{thrm3} to inequality \eqref{eq14} by taking $B=n_1+2$, $$
\eta_{1}:=  \left( \frac{(b-1)\left( 1+\alpha^{n_2-n_1}+\alpha^{n_3-n_1}\right) }{\sqrt{5}d_1}\right), \eta_{2}:= \alpha, \eta_{3}:=b, b_1=1, b_2:= n_1, b_3:=-(\ell_{1}+\ell_{2}+\ell_{3}), s:=3.
$$
With this choice we have $h'(\eta_{1})<3.75\times10^{27}\log^3b\log^2(n_1+2)$ and $\Lambda_3\neq0$. We obtain 
$$
\ell_{1}\log b<1.76\times10^{41}\log^4b\log^3(n_1+2).
$$
\begin{step}\label{step3.4} 
We consider Case 1A and show tat under the assumption that 
	$$
	(n_1-n_2)\log\alpha<1.5\times10^{14}\log^2b\log(n_1+2),\,\,  (n_1-n_3)\log\alpha<7.48\times10^{27}\log^3b\log^2(n_1+2) \text{ and }
	$$  
	
	$$
	\ell_{1}\log b<1.76\times10^{41}\log^4b\log^3(n_1+2),
	$$
	
	we obtain that
	$$
	\ell_{2}\log b< 8.27\times10^{54}\log^5b\log^4(n_1+2).
	$$
\end{step}
Using \eqref{eq2} and \eqref{eq4}, we obtain
\begin{align*}
(b-1)\alpha^{n_{1}}&+(b-1)\alpha^{n_{2}}+(b-1)\alpha^{n_{3}}-\sqrt{5}\left( d_{1}b^{\ell_{1}}-(d_1-d_2)\right)b^{\ell_{2}+\ell_{3}}\\ &= -\sqrt{5}\left((d_2-d_3)b^{\ell_{3}}+d_3\right)+(b-1)\left( \beta^{n_1}+\beta^{n_2}+\beta^{n_3}\right)  
\end{align*}
and the following estimates
\begin{align*}
&\left|(b-1)\alpha^{n_{1}}\left( 1+\alpha^{n_2-n_{1}}+\alpha^{n_3-n_{1}}\right) -\sqrt{5}\left( d_{1}b^{\ell_{1}}-(d_1-d_2)\right)b^{\ell_{2}+\ell_{3}}\right|\\&< \sqrt{5}(b-1)\left(b^{\ell_{3}}+1\right)
 +3(b-1)\alpha^{-n_3}  \\
&< 3.36(b-1)b^{\ell_{3}}.
\end{align*}
Dividing both sides of the above inequality by $ (b-1)\alpha^{n_{1}}\left( 1+\alpha^{n_2-n_{1}}+\alpha^{n_3-n_{1}}\right)$, we get
\begin{align*}
\left| \frac{\sqrt{5}\left( d_1b^{\ell_{1}}-(d_1-d_2)\right) }{(b-1)(1+\alpha^{n_2-n_1}+\alpha^{n_3-n_1})}\alpha^{-n_1}b^{\ell_{2}+\ell_{3}}-1\right|&< \frac{3.36b^{\ell_{3}}}{\alpha^{n_1}\left( 1+\alpha^{n_2-n_{1}}+\alpha^{n_3-n_{1}}\right)}\\
&<\frac{3.36b^{\ell_{3}}}{\alpha^{n_1}}.
\end{align*}
Using \eqref{eq5} we have $\frac{1}{\alpha^{n_{1}}}<\frac{3}{\alpha b^{\ell_{1}+\ell_{2}+\ell_{3}-1}}<\frac{1.86}{b^{\ell_{1}+\ell_{2}+\ell_{3}-1}}.$ Therefore the above inequality implies 
\begin{align*}
\left| \frac{\sqrt{5}\left( d_1b^{\ell_{1}}-(d_1-d_2)\right) }{(b-1)(1+\alpha^{n_2-n_1}+\alpha^{n_3-n_1})}\alpha^{-n_1}b^{\ell_{2}+\ell_{3}}-1\right|&< \frac{6.25b^{-(\ell_{2}-1)}}{b^{\ell_{1}}},
\end{align*}
which leads to
\begin{align}
\label{eq15}
|\Lambda_4|=\left| \frac{\sqrt{5}\left( d_1b^{\ell_{1}}-(d_1-d_2)\right) }{(b-1)(1+\alpha^{n_2-n_1}+\alpha^{n_3-n_1})}\alpha^{-n_1}b^{\ell_{2}+\ell_{3}}-1\right|<3.13b^{-(\ell_{2}-1)} \text{ since } b\ge 2.
\end{align}
We apply Theorem \ref{thrm3} to inequality (\ref{eq15}) by taking $B=n_1+2$, $$
\eta_{1}:=  \frac{\sqrt{5}\left( d_1b^{\ell_{1}}-(d_1-d_2)\right) }{(b-1)(1+\alpha^{n_2-n_1}+\alpha^{n_3-n_1})}, \eta_{2}:= \alpha, \eta_{3}:=b, b_1=1, b_2:= -n_1, b_3:=\ell_{2}+\ell_{3}, s:=3.
$$
Note that $h'(\eta_{1})<1.77\times10^{41}\log^{4}b\log^3(n_1+2)$ and $|\Lambda_4|\neq0$. Therefore,  we get
$$
\ell_{2}\log b< 8.27\times10^{54}\log^5b\log^4(n_1+2).
$$

\begin{step}\label{step3.5}
We now consider Case~1B and show that, under the assumptions
\[
(n_1 - n_2)\log\alpha < 1.5 \times 10^{14} \, \log^2 b \, \log(n_1 + 2), \quad
\ell_1 \log b < 7.48 \times 10^{27} \, \log^3 b \, \log^2(n_1 + 2),
\]
it follows that
\[
\min\left\lbrace \ell_2 \log b, \, (n_1 - n_3)\log \alpha \right\rbrace
< 3.6 \times 10^{41} \, \log^4 b \, \log^3(n_1 + 2).
\]
\end{step}

\noindent
From~\eqref{eq2} and~\eqref{eq4}, we deduce
\begin{align*}
&(b-1)\alpha^{n_1} + (b-1)\alpha^{n_2} - \sqrt{5}\bigl(d_1 b^{\ell_1} - (d_1 - d_2)\bigr) b^{\ell_2 + \ell_3} \\
&\quad = -\sqrt{5}\bigl((d_2 - d_3)b^{\ell_3} + d_3\bigr)
 - \sqrt{5}(b-1) F_{n_3} + (b-1)(\beta^{n_1} + \beta^{n_2}),
\end{align*}
which implies
\begin{align*}
&\Bigl|(b-1)\alpha^{n_1}\bigl( 1 + \alpha^{n_2 - n_1} \bigr) - \sqrt{5}\bigl(d_1 b^{\ell_1} - (d_1 - d_2)\bigr) b^{\ell_2 + \ell_3} \Bigr| \\
&\quad < \sqrt{5}(b-1)(b^{\ell_3} + 1) + \sqrt{5}(b-1)\alpha^{n_3 - 1} + 2(b-1)\alpha^{-n_3} \\
&\quad < 1.5 \sqrt{5} (b-1) b^{\ell_3} + 1.39 (b-1) \alpha^{n_3} \\
&\quad < 4.75 (b-1) \max\{ b^{\ell_3}, \alpha^{n_3} \}.
\end{align*}

\noindent
Dividing both sides by $(b-1)\alpha^{n_1}(1 + \alpha^{n_2 - n_1})$, we obtain
\begin{align*}
\Biggl| \frac{\sqrt{5} \bigl(d_1 b^{\ell_1} - (d_1 - d_2)\bigr)}{(b-1)(1 + \alpha^{n_2 - n_1})} 
\alpha^{-n_1} b^{\ell_2 + \ell_3} - 1 \Biggr| 
&< \max \Bigl\lbrace \frac{4.75 b^{\ell_3}}{\alpha^{n_1}}, \, \alpha^{-(n_1 - n_3)} \Bigr\rbrace.
\end{align*}

\noindent
Applying the bound 
\[
\frac{1}{\alpha^{n_1}} < \frac{1.86}{b^{\ell_1 + \ell_2 + \ell_3 - 1}},
\] 
we then deduce
\begin{align}
\label{eq16}
|\Lambda_5| = \Biggl| \frac{\sqrt{5}\bigl(d_1 b^{\ell_1} - (d_1 - d_2)\bigr)}{(b-1)(1 + \alpha^{n_2 - n_1})} 
\alpha^{-n_1} b^{\ell_2 + \ell_3} - 1 \Biggr|
< 4.42 \, \max \bigl\lbrace b^{-(\ell_2 - 1)}, \, \alpha^{-(n_1 - n_3)} \bigr\rbrace.
\end{align}

\noindent
We apply Theorem~\ref{thrm3} to inequality~\eqref{eq16} with parameters
\[
B = n_1 + 2, \quad
\eta_1 := \frac{\sqrt{5} (d_1 b^{\ell_1} - (d_1 - d_2))}{(b-1)(1 + \alpha^{n_2 - n_1})}, \quad
\eta_2 := \alpha, \quad
\eta_3 := b, \quad
b_1 = 1, \quad
b_2 := -n_1,  
\]
$b_3 := \ell_2 + \ell_3,$ and    $s := 3.$
\noindent
In this case, we have $h'(\eta_1) < 7.5 \times 10^{27} \, \log^3 b \, \log^2(n_1 + 2)$ and $|\Lambda_5| \neq 0$.  
Hence, we conclude that
\[
\min \bigl\lbrace \ell_2 \log b, \, (n_1 - n_3)\log \alpha \bigr\rbrace 
< 3.6 \times 10^{41} \, \log^4 b \, \log^3(n_1 + 2).
\]

We now consider the following sub-sub-cases for Case~1B:

\noindent
\textbf{Case 1B-A:} 
\[
\min\left\lbrace \ell_2 \log b, \, (n_1 - n_3)\log \alpha \right\rbrace = (n_1 - n_3)\log \alpha,
\]

\noindent
\textbf{Case 1B-B:} 
\[
\min\left\lbrace \ell_2 \log b, \, (n_1 - n_3)\log \alpha \right\rbrace = \ell_2 \log b.
\]

\noindent
We will treat these sub-sub-cases separately in the following steps.

\begin{step}\label{step3.6}
We consider Case~1B-A and show that, under the assumptions
\[
(n_1 - n_2)\log \alpha < 1.5 \times 10^{14} \, \log^2 b \, \log(n_1+2), \quad
\ell_1 \log b < 7.48 \times 10^{27} \, \log^3 b \, \log^2(n_1+2),
\]
\[
(n_1 - n_3)\log \alpha < 3.6 \times 10^{41} \, \log^4 b \, \log^3(n_1+2)
\]
it follows that
\[
\ell_2 \log b < 8.7 \times 10^{54} \, \log^5 b \, \log^4(n_1+2).
\]
\end{step}

\noindent
From~\eqref{eq2} and~\eqref{eq4}, we deduce
\begin{align*}
&(b-1)\alpha^{n_1} + (b-1)\alpha^{n_2} + (b-1)\alpha^{n_3} - \sqrt{5} \bigl(d_1 b^{\ell_1} - (d_1 - d_2) \bigr) b^{\ell_2 + \ell_3} \\
&\quad = - \sqrt{5} \bigl((d_2 - d_3) b^{\ell_3} + d_3 \bigr) + (b-1)(\beta^{n_1} + \beta^{n_2} + \beta^{n_3}),
\end{align*}
which implies
\begin{align*}
&\Bigl|(b-1)\alpha^{n_1} \bigl( 1 + \alpha^{n_2 - n_1} + \alpha^{n_3 - n_1} \bigr) - \sqrt{5} \bigl(d_1 b^{\ell_1} - (d_1 - d_2) \bigr) b^{\ell_2 + \ell_3} \Bigr| \\
&\quad < \sqrt{5}(b-1)(b^{\ell_3} + 1) + 3(b-1) \alpha^{-n_3} \\
&\quad < 3.36 \, b^{\ell_3 + 1}.
\end{align*}

\noindent
Dividing both sides by $\sqrt{5} \bigl(d_1 b^{\ell_1} - (d_1 - d_2) \bigr) b^{\ell_2 + \ell_3}$, we obtain
\begin{align*}
\Biggl| \left(\frac{(b-1) \bigl(1 + \alpha^{n_2 - n_1} + \alpha^{n_3 - n_1} \bigr)}{\sqrt{5} \bigl(d_1 b^{\ell_1} - (d_1 - d_2) \bigr)} \right) 
\alpha^{n_1} b^{-(\ell_2 + \ell_3)} - 1 \Biggr|
< \frac{3.36 \, b^{\ell_3 + 1}}{\sqrt{5} \bigl(d_1 b^{\ell_1} - (d_1 - d_2) \bigr) b^{\ell_2 + \ell_3}}.
\end{align*}

\noindent
Hence, we obtain
\begin{align}
\label{eq17}
|\Lambda_6| = \Biggl| \left(\frac{(b-1) \bigl(1 + \alpha^{n_2 - n_1} + \alpha^{n_3 - n_1} \bigr)}{\sqrt{5} \bigl(d_1 b^{\ell_1} - (d_1 - d_2) \bigr)} \right) 
\alpha^{n_1} b^{-(\ell_2 + \ell_3)} - 1 \Biggr| < 1.6 \, b^{-\ell_2 + 1}.
\end{align}

\noindent
Applying Theorem~\ref{thrm3} to inequality~\eqref{eq17} with parameters
\[
B = n_1 + 2, \quad 
\eta_1 := \frac{(b-1)(1 + \alpha^{n_2 - n_1} + \alpha^{n_3 - n_1})}{\sqrt{5} (d_1 b^{\ell_1} - (d_1 - d_2))}, \quad 
\eta_2 := \alpha, \quad 
\eta_3 := b, \quad 
b_1 = 1, \quad 
b_2 := n_1, 
\]
$$
b_3 := -(\ell_2 + \ell_3), \quad 
s := 3,
$$
and noting that $h'(\eta_1) < 1.85 \times 10^{41} \, \log^4 b \, \log^3(n_1+2)$ and $|\Lambda_6| \neq 0$, 
we conclude that
\[
\ell_2 \log b < 8.7 \times 10^{54} \, \log^5 b \, \log^4(n_1+2).
\]

\begin{step}\label{step3.7}
We consider Case~1B-B and show that, under the assumptions
\[
(n_1 - n_2)\log \alpha < 1.5 \times 10^{14} \, \log^2 b \, \log(n_1+3), \quad
\ell_1 \log b < 7.48 \times 10^{27} \, \log^3 b \, \log^2(n_1+2),
\]
$$
 \ell_2 \log b < 3.6 \times 10^{41} \, \log^4 b \, \log^3(n_1+2),
$$
it follows that
\[
(n_1 - n_3)\log \alpha < 8.7 \times 10^{54} \, \log^5 b \, \log^4(n_1+2).
\]
\end{step}

\noindent
From~\eqref{eq2} and~\eqref{eq4}, we deduce
\begin{align*}
&(b-1)\alpha^{n_1} + (b-1)\alpha^{n_2} - \sqrt{5} \bigl(d_1 b^{\ell_1 + \ell_2} - (d_1 - d_2) b^{\ell_2} - (d_2 - d_3) \bigr) b^{\ell_3} \\
&\quad = -\sqrt{5} d_3 - \sqrt{5}(b-1) F_{n_3} + (b-1)(\beta^{n_1} + \beta^{n_2}),
\end{align*}
which implies
\begin{align*}
&\Bigl| (b-1)\alpha^{n_1} (1 + \alpha^{n_2 - n_1}) - \sqrt{5} \bigl(d_1 b^{\ell_1 + \ell_2} - (d_1 - d_2) b^{\ell_2} - (d_2 - d_3) \bigr) b^{\ell_3} \Bigr| \\
&\quad < \sqrt{5} (b-1) + \sqrt{5} (b-1) \alpha^{n_3 - 1} + 2 (b-1) \alpha^{-n_3} \\
&\quad < 1.4 (b-1) \alpha^{n_3}.
\end{align*}

\noindent
Dividing both sides by $(b-1) \alpha^{n_1} (1 + \alpha^{n_2 - n_1})$, we obtain
\begin{align}
\label{eq18}
|\Lambda_7| = \Biggl| \frac{\sqrt{5} \bigl(d_1 b^{\ell_1 + \ell_2} - (d_1 - d_2) b^{\ell_2} - (d_1 - d_3) \bigr)}{(b-1)(1 + \alpha^{n_2 - n_1})} 
\alpha^{-n_1} b^{\ell_3} - 1 \Biggr| < 1.4 \, \alpha^{-(n_1 - n_3)}.
\end{align}

\noindent
Applying Theorem~\ref{thrm3} to inequality~\eqref{eq18} with parameters
\[
B = n_1 + 2, \quad 
\eta_1 := \frac{\sqrt{5} (d_1 b^{\ell_1 + \ell_2} - (d_1 - d_2) b^{\ell_2} - (d_1 - d_3))}{(b-1)(1 + \alpha^{n_2 - n_1})}, \quad
\eta_2 := \alpha, \quad
\eta_3 := b, \quad
b_1 = 1, 
\]
$$
b_2 := -n_1, \quad   b_3 := \ell_3, \quad
s := 3,
$$
and noting that $h'(\eta_1) < 1.85 \times 10^{41} \, \log^4 b \, \log^3(n_1+2)$ and $|\Lambda_7| \neq 0$, we conclude that
\[
(n_1 - n_3)\log \alpha < 8.7 \times 10^{54} \, \log^5 b \, \log^4(n_1+2).
\]

\begin{step}
	\label{step3.8} 
We consider Case 2 and show that under the assumption that 
$$
\ell_{1}\log b<1.5\times10^{14}\log^2b\log(n_1+2),
$$

 we obtain that
 
$$
\min\{(n_1-n_2)\log\alpha, \ell_{2}\log b\}<7.48\times10^{27}\log^3b\log^2(n_1+2).
$$
\end{step}
Using (\ref{eq2}) and (\ref{eq4}), we obtain
\begin{align*}
(b-1)\alpha^{n_{1}}-\sqrt{5}\left( d_{1}b^{\ell_{1}}-(d_1-d_2)\right)b^{\ell_{2}+\ell_{3}} &= -\sqrt{5}\left((d_2-d_3)b^{\ell_{3}}+d_3\right)\\
&-\sqrt{5}(b-1)(F_{n_2}+F_{n_3})+(b-1) \beta^{n_1}  \\
\left|(b-1)\alpha^{n_{1}} -\sqrt{5}\left( d_{1}b^{\ell_{1}}-(d_1-d_2)\right)b^{\ell_{2}+\ell_{3}}\right|&< \sqrt{5}(b-1)\left(b^{\ell_{3}}+1\right)\\
& +2\sqrt{5}(b-1)\alpha^{n_2-1}+(b-1)\alpha^{-n_1}  \\
&< 1.5\sqrt{5}(b-1)b^{\ell_{3}}+2.77(b-1)\alpha^{n_2}\\
&<6.13(b-1)\max\{b^{\ell_{3}}, \alpha^{n_2}\}
\end{align*}
Dividing both sides of the above inequality by $ (b-1)\alpha^{n_{1}}$, we get
\begin{align*}
\left|\frac{\sqrt{5}\left( d_1b^{\ell_{1}}-(d_1-d_2)\right)}{b-1}.\alpha^{-n_1}.b^{\ell_{2}+\ell_{3}}-1  \right|&< \max\left\lbrace \frac{6.13b^{\ell_{3}}}{\alpha^{n_1}},\alpha^{-(n_1-n_2)}\right\rbrace. \\
\end{align*}
Using the fact that  $\frac{1}{\alpha^{n_{1}}}<\frac{1.86}{b^{\ell_{1}+\ell_{2}+\ell_{3}-1}}$, we obtain 
\begin{align}
\label{eq19}
\left|\Lambda'_{1} \right| =\left|\frac{\sqrt{5}\left( d_1b^{\ell_{1}}-(d_1-d_2)\right)}{b-1}.\alpha^{-n_1}.b^{\ell_{2}+\ell_{3}}-1  \right|< 5.7\max\left\lbrace b^{-(\ell_{2}-1)}, \alpha^{-(n_1-n_2)}\right\rbrace 
\end{align}
We apply Theorem \ref{thrm3} to inequality (\ref{eq19}) by taking $B=n_1+2$
$$
\eta_{1}:=  \left(\frac{\sqrt{5}\left( d_1b^{\ell_{1}}-(d_1-d_2)\right)}{b-1} \right), \eta_{2}:= \alpha, \eta_{3}:=b, b_1=1, b_2:= -n_1, b_3:=\ell_{3}, s:=3.
$$
In this case we have $h'(\eta_{1})<1.6\times10^{14}\log^{2}b\log(n_1+2).$ and $|\Lambda'_1|\neq0$. \\ Therefore,  we get
$$
\min\{(n_1-n_2)\log\alpha, \ell_{2}\log b\}<7.48\times10^{27}\log^3b\log^2(n_1+2).
$$
At this stage, we consider the following two sub-cases for Case~2:

\noindent
 \textbf{Case 2A:}
\[
  \min\{(n_1-n_2)\log \alpha, \, \ell_2 \log b\} = \ell_2 \log b
\]
\noindent
     \textbf{Case 2B:} 
\[     
     \min\{(n_1-n_2)\log \alpha, \, \ell_2 \log b\} = (n_1-n_2) \log \alpha.
 \]    
We will analyze these sub-cases in the subsequent steps.

\begin{step}
	\label{step3.9} 
We consider Case 2A and show that under the assumption that 
$$
\ell_{1}<1.5\times10^{14}\log^2b\log(n_1+2) \text{ and } \ell_{2}\log\alpha<7.48\times10^{27}\log^3b\log^2(n_1+2),
$$
 we obtain that
$$
(n_1-n_2)\log\alpha<1.76\times10^{41}\log^4b\log^3(n_1+2). 
$$
\end{step}
Using \eqref{eq2} and \eqref{eq4}, we obtain
\begin{align*}
(b-1)\alpha^{n_{1}} 
&- \sqrt{5}\Bigl( d_{1}b^{\ell_{1}+\ell_{2}} - (d_1-d_2)b^{\ell_{2}} - (d_2-d_3) \Bigr) b^{\ell_{3}} \\
&= -\sqrt{5}\, d_3 - \sqrt{5}(b-1)(F_{n_2}+F_{n_3}) + (b-1)\beta^{n_1}, \\[2mm]
\Bigl|(b-1)\alpha^{n_{1}} 
&- \sqrt{5}\Bigl( d_{1}b^{\ell_{1}+\ell_{2}} - (d_1-d_2)b^{\ell_{2}} - (d_2-d_3) \Bigr) b^{\ell_{3}}\Bigr| \\
&< \sqrt{5}(b-1) + 2\sqrt{5}(b-1)\alpha^{n_2-1} + (b-1)\alpha^{-n_1}, \\[1mm]
&< 2.77 (b-1) \alpha^{n_2}.
\end{align*}
Dividing both sides of the above inequality by $ (b-1)\alpha^{n_{1}}$, we get
\begin{align}
\label{eq20}
\left|\Lambda'_{2} \right| =\left|\frac{\sqrt{5}\left( d_1b^{\ell_{1}+\ell_{2}}-(d_1-d_2)b^{\ell_{2}}-(d_2-d_3)\right)}{b-1}.\alpha^{-n_1}.b^{\ell_{3}}-1  \right|< 2.77\alpha^{-(n_1-n_2)}
\end{align}
We apply Theorem \ref{thrm3} to inequality \eqref{eq20} by taking $B=n_1+2$
$$
\eta_{1}:=  \left(\frac{\sqrt{5}\left( d_1b^{\ell_{1}+\ell_{2}}-(d_1-d_2)b^{\ell_{2}}-(d_2-d_3)\right)}{b-1}\right), \eta_{2}:= \alpha, \eta_{3}:=b, b_1=1, b_2:= -n_1, 
$$
$ b_3:=\ell_{3}, s:=3.$ In this case we have $h'(\eta_{1})<3.75\times10^{27}\log^{3}b\log^2(n_1+2)$ and $|\Lambda'_2|\neq0$. Therefore,  we get
$$
(n_1-n_2)\log\alpha<1.76\times10^{41}\log^4b\log^3(n_1+2). 
$$

\begin{step}
		\label{step3.10} 
	We consider Case 2A and show that under the assumption that 
	$$
	\ell_{1}<1.5\times10^{14}\log^2b\log(n_1+2), \,\,\, \ell_{2}\log\alpha<7.48\times10^{27}\log^3b\log^2(n_1+2) \text{ and }
	$$ 
	
	$$
	(n_1-n_2)\log\alpha<1.76\times10^{41}\log^4b\log^3(n_1+2),
	$$
	we obtain that
	$$
	(n_1-n_3)\log\alpha<8.27\times10^{54}\log^5b\log^4(n_1+2). 
	$$
\end{step}
Using \eqref{eq2} and \eqref{eq4}, we obtain
\begin{align*}
(b-1)\alpha^{n_{1}} + (b-1)\alpha^{n_{2}}
&- \sqrt{5} \Bigl( d_{1}b^{\ell_{1}+\ell_{2}} - (d_1-d_2)b^{\ell_{2}} - (d_2-d_3) \Bigr) b^{\ell_{3}} \\
&= -\sqrt{5}\, d_3 - \sqrt{5}(b-1) F_{n_3} + (b-1)(\beta^{n_1} + \beta^{n_2}), \\[1mm]
\Bigl| (b-1)\alpha^{n_{1}} \bigl( 1 + \alpha^{n_2-n_1} \bigr) 
&- \sqrt{5} \Bigl( d_{1}b^{\ell_{1}+\ell_{2}} - (d_1-d_2)b^{\ell_{2}} - (d_2-d_3) \Bigr) b^{\ell_{3}} \Bigr| \\
&< \sqrt{5}(b-1) + \sqrt{5}(b-1)\alpha^{n_3-1} + 2(b-1)\alpha^{-n_3}, \\[1mm]
&< 1.4 (b-1) \alpha^{n_3}.
\end{align*}

Dividing both sides of the above inequality by $(b-1)\alpha^{n_1}\left( 1 + \alpha^{n_2-n_1} \right)$, we get
\begin{align}
\label{eq21}
|\Lambda'_3|=\left|\frac{\sqrt{5}\left(d_1b^{\ell_{1}+\ell_{2}}-(d_1-d_2)b^{\ell_{2}}-(d_1-d_3) \right) }{(b-1)\left( 1+\alpha^{n_2-n_1}\right) }.\alpha^{-n_1}.b^{\ell_{3}} -1 \right|<1.4.b^{-(n_1-n_3)}.
\end{align}
and therefore
$$
(n_1-n_3)\log\alpha<8.27\times10^{54}\log^5b\log^4(n_1+2). 
$$
\begin{step}
\label{step3.11}
We consider Case 2B and prove that, under the assumptions
$$
\ell_{1}<1.5\times10^{14}\log^2 b\,\log(n_1+2)
\quad \text{and} \quad
(n_1-n_2)\log\alpha<7.48\times10^{27}\log^3 b\,\log^2(n_1+2),
$$
one obtains the bound
$$
\min\{(n_1-n_3)\log\alpha,\ \ell_{2}\log b\}
<3.6\times10^{41}\log^4 b\,\log^3(n_1+2).
$$
\end{step}

We note that Step \ref{step3.11} is analogous to Step \ref{step3.5}, with the roles of $\ell_{1}\log b$ and $(n_1-n_2)\log\alpha$ interchanged. We now divide Case 2B into the following sub-subcases.

\noindent
\textbf{Case 2B-A:}
\[
  \min\{(n_1-n_3)\log\alpha,\ \ell_{2}\log b\}=(n_1-n_3)\log\alpha
\]
\noindent
\textbf{Case 2B-B:}
\[
  \min\{(n_1-n_3)\log\alpha,\ \ell_{2}\log b\}=\ell_{2}\log b.
\]
These sub-subcases will be treated in the subsequent steps. By applying the same arguments as in Steps \ref{step3.6} and \ref{step3.7}, we obtain the results stated in Steps \ref{step3.12} and \ref{step3.13}.

\begin{step}
\label{step3.12}
We consider Case 2B-A and prove that, under the assumptions
$$
\ell_{1}\log b<1.5\times10^{14}\log^2 b\,\log(n_1+2), \qquad
(n_1-n_2)\log\alpha<7.48\times10^{27}\log^3 b\,\log^2(n_1+2),
$$
and
$$
(n_1-n_3)\log\alpha<3.6\times10^{41}\log^4 b\,\log^3(n_1+2),
$$
one obtains
\begin{align*}
\ell_{2}\log b< 8.7\times10^{54}\log^5 b\,\log^4(n_1+2).
\end{align*}

We apply again Theorem \ref{thrm3} to derive an upper bound for $\ell_{2}\log b$. 
The argument is entirely analogous to Case 1B-A. In particular, we obtain
\begin{align*}
|\Lambda_6|
=\left|
\left(\frac{(b-1)\bigl(1+\alpha^{n_2-n_1}+\alpha^{n_3-n_1}\bigr)}
{\sqrt{5}\bigl(d_1b^{\ell_{1}}-(d_1-d_2)\bigr)}\right)
\alpha^{n_1}b^{\ell_{2}+\ell_{3}}-1
\right|
<1.6\,b^{-\ell_{2}+1},
\end{align*}
with the same parameter setting as in Case 1B-A. Consequently, 
Theorem \ref{thrm3} yields
\begin{align*}
\ell_{2}\log b< 8.7\times10^{54}\log^5 b\,\log^4(n_1+2).
\end{align*}
\end{step}

\begin{step}
\label{step3.13}
We consider Case 2B-B and prove that, under the assumptions
$$
\ell_{1}\log b<1.5\times10^{14}\log^2 b\,\log(n_1+2), \qquad
(n_1-n_2)\log\alpha<7.48\times10^{27}\log^3 b\,\log^2(n_1+2),
$$
and
$$
\ell_{2}\log b<3.6\times10^{41}\log^4 b\,\log^3(n_1+2),
$$
one obtains
\begin{align*}
(n_1-n_3)\log\alpha
<8.7\times10^{54}\log^5 b\,\log^4(n_1+2).
\end{align*}

We again apply Theorem \ref{thrm3} to derive an upper bound for
$(n_1-n_3)\log\alpha$. The argument is entirely analogous to that of
Case~1B-B. In particular, we obtain
\begin{align*}
\left|
\frac{\sqrt{5}\bigl(d_1b^{\ell_{1}+\ell_{2}}
-(d_1-d_2)b^{\ell_{2}}-(d_1-d_3)\bigr)}
{(b-1)\bigl(1+\alpha^{n_2-n_1}\bigr)}
\alpha^{-n_1}b^{\ell_{3}}-1
\right|
<1.4\,\alpha^{-(n_1-n_3)},
\end{align*}
with the same parameter setting as in Case~1B-B. Consequently,
Theorem~\ref{thrm3} yields
\begin{align*}
(n_1-n_3)\log\alpha
<8.7\times10^{54}\log^5 b\,\log^4(n_1+2).
\end{align*}
\end{step}

Table~\ref{table1} summarizes the bounds obtained so far.

\begin{table}[htbp]
\centering
\footnotesize
\renewcommand{\arraystretch}{1.25}

\begin{tabular}{|c|c|c|c|}
\hline
\textbf{Upper bound of}
& \textbf{Case 1A}
& \textbf{Case 1B-A}
& \textbf{Case 1B-B} \\
\hline
$(n_1-n_2)\log\alpha$
& $1.5\times10^{14}\log^2 b\,\log(n_1+2)$
& $1.5\times10^{14}\log^2 b\,\log(n_1+2)$
& $1.5\times10^{14}\log^2 b\,\log(n_1+2)$ \\
\hline
$(n_1-n_3)\log\alpha$
& $7.48\times10^{27}\log^3 b\,\log^2(n_1+2)$
& $8.7\times10^{54}\log^5 b\,\log^4(n_1+2)$
& $3.6\times10^{41}\log^4 b\,\log^3(n_1+2)$ \\
\hline
$\ell_{1}\log b$
& $1.7\times10^{41}\log^4 b\,\log^3(n_1+2)$
& $7.48\times10^{27}\log^3 b\,\log^2(n_1+2)$
& $7.48\times10^{27}\log^3 b\,\log^2(n_1+2)$ \\
\hline
$\ell_{2}\log b$
& $8.27\times10^{54}\log^5 b\,\log^4(n_1+2)$
& $3.6\times10^{41}\log^4 b\,\log^3(n_1+2)$
& $8.7\times10^{54}\log^5 b\,\log^4(n_1+2)$ \\
\hline
\end{tabular}

\vspace{0.4cm}

\begin{tabular}{|c|c|c|c|}
\hline
\textbf{Upper bound of}
& \textbf{Case 2A}
& \textbf{Case 2B-A}
& \textbf{Case 2B-B} \\
\hline
$(n_1-n_2)\log\alpha$
& $1.76\times10^{41}\log^4 b\,\log^3(n_1+2)$
& $7.48\times10^{27}\log^3 b\,\log^2(n_1+2)$
& $7.48\times10^{27}\log^3 b\,\log^2(n_1+2)$ \\
\hline
$(n_1-n_3)\log\alpha$
& $8.27\times10^{54}\log^5 b\,\log^4(n_1+2)$
& $3.6\times10^{41}\log^4 b\,\log^3(n_1+2)$
& $8.7\times10^{54}\log^5 b\,\log^4(n_1+2)$ \\
\hline
$\ell_{1}\log b$
& $1.5\times10^{14}\log^2 b\,\log(n_1+2)$
& $1.5\times10^{14}\log^2 b\,\log(n_1+2)$
& $1.5\times10^{14}\log^2 b\,\log(n_1+2)$ \\
\hline
$\ell_{2}\log b$
& $7.48\times10^{27}\log^3 b\,\log^2(n_1+2)$
& $8.7\times10^{54}\log^5 b\,\log^4(n_1+2)$
& $3.6\times10^{41}\log^4 b\,\log^3(n_1+2)$ \\
\hline
\end{tabular}

\caption{Summary of the upper bounds obtained in all cases}
\label{table1}
\end{table}

\begin{step}
\label{step3.14}
Assuming that the bounds given in Table~\ref{table1} hold, we prove that
\[
n_1<1.51\times10^{82}\log^{11} b.
\]
\end{step}

\begin{proof}
Using \eqref{eq2} and \eqref{eq4}, we obtain
{\footnotesize
\begin{align*}
&(b-1)\alpha^{n_{1}}+(b-1)\alpha^{n_{2}}+(b-1)\alpha^{n_{3}}
-\sqrt{5}\left(d_{1}b^{\ell_{1}+\ell_{2}}
-(d_1-d_2)b^{\ell_{2}}-(d_2-d_3)\right)b^{\ell_{3}}  \\
&\qquad= -\sqrt{5}d_3+(b-1)(\beta^{n_1}+\beta^{n_2}+\beta^{n_3}),
\\[0.2cm]
&\Bigl|(b-1)\alpha^{n_{1}}\bigl(1+\alpha^{n_2-n_1}+\alpha^{n_3-n_1}\bigr)
-\sqrt{5}\left(d_{1}b^{\ell_{1}+\ell_{2}}
-(d_1-d_2)b^{\ell_{2}}-(d_2-d_3)\right)b^{\ell_{3}}\Bigr|  \\
&\qquad< \sqrt{5}(b-1)+3(b-1)\alpha^{-n_3}
<2.24(b-1).
\end{align*}
}

Dividing both sides of the above inequality by
\[
(b-1)\alpha^{n_{1}}\bigl(1+\alpha^{n_2-n_1}+\alpha^{n_3-n_1}\bigr),
\]
we obtain
\begin{align}
\label{eq22}
|\Lambda|=
\left|
\frac{\sqrt{5}\left(d_1b^{\ell_{1}+\ell_{2}}
-(d_1-d_2)b^{\ell_{2}}-(d_2-d_3)\right)}
{(b-1)\left(1+\alpha^{n_2-n_1}+\alpha^{n_3-n_1}\right)}
\alpha^{-n_1}b^{\ell_{3}}-1
\right|
<2.24\,\alpha^{-n_1}.
\end{align}

We now apply Theorem~\ref{thrm3} to inequality~\eqref{eq22} with
\[
B=n_1+2,
\quad
\eta_{1}:=
\frac{\sqrt{5}\left(d_1b^{\ell_{1}+\ell_{2}}
-(d_1-d_2)b^{\ell_{2}}-(d_2-d_3)\right)}
{(b-1)\left(1+\alpha^{n_2-n_1}+\alpha^{n_3-n_1}\right)},
\]
\[
\eta_{2}:=\alpha,\qquad
\eta_{3}:=b,\qquad
b_1=1,\quad b_2:=-n_1,\quad b_3:=\ell_{3},\quad s:=3.
\]
In this case we have
\[
h'(\eta_{1})
<8.28\times10^{54}\log^{5} b\,\log^{4}(n_1+2),
\qquad |\Lambda|\neq0.
\]
Therefore,
\[
n_1\log\alpha
<3.86\times10^{68}\log^{6} b\,\log^{5}(n_1+2).
\]
Using the inequality $\log(n_1+2)<2\log n_1$, we obtain
\[
n_1<1.6\times10^{69}\log^{6} b\,\log^{5} n_1.
\]
Next, we apply Lemma~\ref{lem2} with parameters
\[
\ell:=5,\qquad
H:=1.6\times10^{69}\log^{6} b,\qquad
L:=n_1.
\]
Hence,
\[
n_1<
2^{5}\times1.6\times10^{69}\log^{6} b
\left(159.35+6\log\log b\right)^{5}.
\]
This implies
\[
n_1<3.09\times10^{82}\log^{11} b.
\]
Since $b\ge2$, we have
\[
159.35+6\log\log b<227\log b,
\]
and the desired bound follows. Therefore,
Theorem~\ref{thrm1} is proved.
\end{proof}

\section{Proof of theorem \ref{thrm2}}
In the final step, we reduce the large upper bound for $n_1$ obtained in
Theorem~\ref{thrm1} by repeatedly applying Lemma~\ref{lem3}. More precisely,
we explicitly determine all sums of three Fibonacci numbers that are
concatenations of three repdigits in base $b$, for $2\le b\le10$.
Within this range, we deduce that every solution of equation~\eqref{eq1}
satisfies
\[
n_1<3\times10^{86}.
\]

We first consider inequality~\eqref{eq7} and recall that
\[
\Gamma_1
=\log\left(\frac{b-1}{d_1\sqrt{5}}\right)
+n_1\log\alpha
-(\ell_{1}+\ell_{2}+\ell_{3})\log b.
\]
For technical reasons, we assume that
\[
\min\{n_1-n_2,\; n_1-n_3,\; \ell_{1},\; \ell_{2}\}\ge24.
\]
If this condition is not satisfied, we proceed according to the following cases:
\begin{enumerate}[$\bullet$]
\item If $n_1-n_2<24$ while $\ell_{1}, \ell_{2}, n_1-n_2 \ge 24$, then we consider inequality~(\ref{eq10}), that is, we move to Step~\ref{step3.2}.

\item If $n_1-n_2<24$ and $n_1-n_3<24$ while $\ell_{1}, \ell_{2}\ge 24$, then we consider inequality~(\ref{eq13}), that is, we move to Step~\ref{step3.3}; afterwards, we consider inequality~(\ref{eq15}), that is, we proceed to Step~\ref{step3.4}.

\item If $n_1-n_2<24$ and $\ell_{1}<24$ while $\ell_{2}, n_1-n_3\ge 24$, then we consider inequality~(\ref{eq16}), that is, we move to Step~\ref{step3.5}, and afterwards proceed to Step~\ref{step3.11}.

\item If $n_1-n_2<24$, $n_1-n_3<24$, and $\ell_{1}<24$ while $\ell_{2}\ge 24$, then we consider inequality~(\ref{eq17}), that is, we move to Step~\ref{step3.6}, and afterwards proceed to Step~\ref{step3.12}.

\item If $n_1-n_2<24$, $\ell_{1}<24$, and $\ell_{2}<24$ while $n_1-n_3\ge 24$, then we consider inequality~(\ref{eq18}), leading to Step~\ref{step3.7} and Step~\ref{step3.13}.

\item If $\ell_{1}<24$ while $n_1-n_2, n_1-n_3, \ell_{2}\ge 24$, then we consider inequality~(\ref{eq19}), which leads to Step~\ref{step3.8}.

\item If $\ell_{1}<24$ and $\ell_{2}<24$ while $n_1-n_2, n_1-n_3\ge 24$, then we consider inequality~(\ref{eq20}), that is, we move to Step~\ref{step3.9}; afterwards, we consider inequality~(\ref{eq21}), that is, we proceed to Step~\ref{step3.10}.

\item If all $n_1-n_2, n_1-n_3, \ell_{1}, \ell_{2}<24$, then we consider inequality~(\ref{eq22}), which leads to Step~\ref{step3.14}.
\end{enumerate}

\begin{step}\label{step4.1}
We prove that $\ell_{1}\le 301$ or $n_1-n_2\le 438$.
\end{step}

We begin by considering inequality~(\ref{eq7}). Since we assume that 
$\min\{n_1-n_2, n_1-n_3, \ell_{1}, \ell_{2}\}\ge 24$, we obtain
$\left| \Lambda_1\right| = \left|\mathrm{e}^{\Gamma_1}-1 \right|<\frac{1}{4}$,
hence $\left|\Lambda_1\right| <\frac{1}{2}$.
Moreover, since the inequality 
$\left| x\right| < 2\left| \mathrm{e}^{x}-1\right|$
holds for all $x$ such that $\left| x\right|<\frac{1}{2}$, we deduce that
\[
\left|\Gamma_1 \right|
<15.7b\max\{b^{-\ell_{1}}, \alpha^{-(n_1-n_2)}\}.
\]
Consequently, we obtain
\[
0<
\left|
(\ell_{1}+\ell_{2}+\ell_{3})
\left( \frac{\log b}{\log\alpha}\right)
-n_1
+\frac{\log\left( \frac{d_1\sqrt{5}}{b-1}\right)}{\log\alpha}
\right|
<
\max\left\{
32.63b^{-(\ell_{1}-1)},\,
32.63b\, \alpha^{-(n_1-n_2)}
\right\}.
\]

To apply Lemma~\ref{lem3} to the above inequality, we choose
\[
\tau:= \frac{\log b}{\log\alpha}, \quad
\mu:= \frac{\log\left( \frac{d_1\sqrt{5}}{b-1}\right)}{\log\alpha},
\]
\[
(A, B)=(32.63, b) \ \text{or}\ (32.63b, \alpha),
\quad
w:= \ell_{1}-1 \ \text{or}\ n_1-n_2,
\quad \text{with } 1\le d_1 \le b-1.
\]

Since $\ell_{1}+\ell_{2}+\ell_{3}<n_1+2<3.1\times10^{86}$,
we may take $M:=3.1\times10^{86}$.
For the remainder of the proof, we use \textit{Mathematica} to apply Lemma~\ref{lem3}.
In the computations, if the first convergent $q_t$ satisfying $q_t>6M$
does not fulfill the condition $\varepsilon>0$,
we consider the next convergent until the condition is satisfied.
Thus we obtain Table~\ref{table2}.

\begin{table}[H]
\centering
\caption{}
\tiny{
\label{table2}
\begin{tabular}{|c|c|c|c|c|c|c|c|c|c|}
\hline
$b$ & 2 & 3 & 4 & 5 & 6 & 7 & 8 & 9 & 10 \\ \hline
$q_t$ & $q_{172}$ & $q_{164}$ & $q_{176}$ & $q_{171}$ & $q_{178}$ & $q_{170}$ & $q_{162}$ & $q_{158}$ & $q_{175}$ \\ \hline
$\varepsilon\ge$ & 0.373 & 0.264 & 0.254 & 0.337 & 0.268 & 0.283 & 0.051 & 0.182 & 0.436 \\ \hline
$\ell_{1}-1$ & 300 & 189 & 150 & 129 & 116 & 107 & 101 & 95 & 90 \\ \hline
$n_1-n_2\le$ & 434 & 434 & 434 & 432 & 433 & 434 & 438 & 435 & 435 \\ \hline
\end{tabular}
}
\end{table}

Therefore,
\[
1\le \ell_{1}\le
\frac{\log\left(32.63q_{172}/0.373 \right)}{\log 2}+1
\le 301,
\]
or
\[
1\le n_1-n_2\le
\frac{\log\left(261.04q_{162}/0.051 \right)}{\log \alpha}+1
\le 438.
\]
Consequently, we distinguish the following cases:

\noindent
\textbf{Case 1:}
\[
  n_1-n_2\le 438
\]
\noindent
\textbf{Case 2:} 
\[
\ell_{1}\le 301.
\]
\begin{step}
We consider Case~1 and prove that, under the assumption $n_1-n_2\le 438$, one has
$n_1-n_3\le 456$ or $\ell_{1}\le 308$.
\end{step}

In this step, we consider inequality~(\ref{eq10}) and assume that
$\ell_{1},\, n_1-n_3 \ge 24$.
Recall that
\[
\Gamma_2=
\log\left(
\frac{(b-1)\left(1+\alpha^{\,n_2-n_1}\right)}{\sqrt{5}\,d_1}
\right)
+n_1\log\alpha
-(\ell_{1}+\ell_{2}+\ell_{3})\log b .
\]
From inequality~(\ref{eq10}) we obtain
\[
|\Gamma_2|
<19.86b\max\left\{b^{-\ell_{1}},\, \alpha^{-(n_1-n_3)}\right\}.
\]
Hence,
\begin{align}
\label{eq23}
0<&\left|
(\ell_{1}+\ell_{2}+\ell_{3})
\left(\frac{\log b}{\log\alpha}\right)
-n_1
+\frac{\log\!\left(
\frac{\sqrt{5}\,d_1}{(b-1)\left(1+\alpha^{-(n_1-n_2)}\right)}
\right)}{\log\alpha}
\right|  \\
&<
\max\left\{
41.28\,b^{-(\ell_{1}-1)},
\,41.28\,b\,\alpha^{-(n_1-n_3)}
\right\}.
\nonumber
\end{align}

Since the hypotheses of Lemma~\ref{lem3} are satisfied,
we apply it to inequality~(\ref{eq23}) with
\[
\tau:=\frac{\log b}{\log\alpha}, \qquad
\mu:=\mu_k=
\frac{\log\!\left(
\frac{\sqrt{5}\,d_1}{(b-1)\left(1+\alpha^{-k}\right)}
\right)}{\log\alpha},
\]
\[
(A,B)=(41.28,b)\ \text{or}\ (41.28b,\alpha),
\qquad
w:=\ell_{1}-1\ \text{or}\ n_1-n_3,
\]
with $1\le d_1\le b-1$ and $k=n_1-n_2\in[0,438]$.
Since $\ell_{1}+\ell_{2}+\ell_{3}<n_1+2<3.1\times10^{86}$,
we take $M:=3.1\times10^{86}$.
Using \textit{Mathematica}, we obtain the following results.

\begin{table}[H]
\centering
\caption{Results for $(A,B)=(41.28,b)$}
\label{table3}
\tiny
\renewcommand{\arraystretch}{1.3}
\resizebox{\textwidth}{!}{
\begin{tabular}{|c|c|c|c|c|c|c|c|c|c|}
\hline
$b$ & 2 & 3 & 4 & 5 & 6 & 7 & 8 & 9 & 10 \\ \hline
$q_t$ & $q_{172}$ & $q_{164}$ & $q_{176}$ & $q_{171}$ & $q_{178}$ & $q_{170}$ & $q_{162}$ & $q_{158}$ & $q_{175}$ \\ \hline
$\varepsilon\ge$ &
$3.79\times10^{-3}$ &
$7.41\times10^{-4}$ &
$8.73\times10^{-4}$ &
$5.69\times10^{-6}$ &
$1.16\times10^{-4}$ &
$6.06\times10^{-4}$ &
$2.98\times10^{-5}$ &
$7.48\times10^{-5}$ &
$4.85\times10^{-5}$ \\ \hline
$\ell_{1}-1\le$ & 307 & 195 & 154 & 136 & 120 & 110 & 105 & 98 & 94 \\ \hline
\end{tabular}
}
\end{table}

\begin{table}[H]
\centering
\caption{Results for $(A,B)=(41.28b,\alpha)$}
\label{table4}
\tiny
\renewcommand{\arraystretch}{1.3}
\resizebox{\textwidth}{!}{
\begin{tabular}{|c|c|c|c|c|c|c|c|c|c|}
\hline
$b$ & 2 & 3 & 4 & 5 & 6 & 7 & 8 & 9 & 10 \\ \hline
$q_t$ & $q_{172}$ & $q_{164}$ & $q_{176}$ & $q_{171}$ & $q_{178}$ & $q_{170}$ & $q_{162}$ & $q_{158}$ & $q_{175}$ \\ \hline
$\varepsilon\ge$ &
$3.79\times10^{-3}$ &
$7.41\times10^{-4}$ &
$8.73\times10^{-4}$ &
$5.69\times10^{-6}$ &
$1.16\times10^{-4}$ &
$3.01\times10^{-4}$ &
$2.98\times10^{-5}$ &
$7.48\times10^{-5}$ &
$4.85\times10^{-5}$ \\ \hline
$n_1-n_3\le$ & 444 & 447 & 447 & 456 & 449 & 449 & 454 & 452 & 455 \\ \hline
\end{tabular}
}
\end{table}

In all cases, we conclude that
\[
1\le \ell_{1}\le
\frac{\log\left(41.28q_{172}/(3.79\times10^{-3})\right)}{\log 2}+1
\le 308,
\]
or
\[
1\le n_1-n_3\le
\frac{\log\left(206.4q_{175}/(4.85\times10^{-5})\right)}{\log\alpha}
\le 456.
\]

Within Case~1, we distinguish the following subcases:

\noindent
\textbf{Case 1A:}
\[
  n_1-n_3\le 456;
 \]
 \textbf{Case 1B:}
 \[
 \ell_{1}\le 308.
\]
\begin{step}
\label{step4.3}
We consider Case~1A and prove that, under the assumptions
$n_1-n_2\le 438$ and $n_1-n_3\le 456$, one has $\ell_{1}\le 314$.
\end{step}

In this step we assume that $\ell_{1}\ge 24$. 
From inequality~\eqref{eq13} we obtain
\[
0<
\left|
(\ell_{1}+\ell_{2}+\ell_{3})
\left(\frac{\log b}{\log\alpha}\right)
-n_1
+\frac{\log\left(
\frac{\sqrt{5}d_1}
{(b-1)\left(1+\alpha^{-(n_1-n_2)}+\alpha^{-(n_1-n_3)}\right)}
\right)}{\log\alpha}
\right|
<7.49b^{-(\ell_{1}-1)} .
\]
We apply Lemma~\ref{lem3} with
\[
\tau:=\frac{\log b}{\log\alpha},\quad
\mu:=\mu_{k,m}=
\frac{\log\left(
\frac{\sqrt{5}d_1}{(b-1)(1+\alpha^{-k}+\alpha^{-m})}
\right)}{\log\alpha},
\]
\[
(A,B)=(7.49,b),\qquad
w:=\ell_{1}-1,
\]
with $1\le d_1\le b-1$,
$k=n_1-n_2\in[0,438]$ and $m=n_1-n_3\in[0,456]$,
subject to $n_1-n_2\le n_1-n_3$.
We take $M:=3.1\times10^{86}$.
Using \textit{Mathematica}, we obtain:

\begin{table}[H]
\centering
\caption{}
\tiny
\renewcommand{\arraystretch}{1.3}
\resizebox{\textwidth}{!}{
\begin{tabular}{|c|c|c|c|c|c|c|c|c|c|}
\hline
$b$ & 2 & 3 & 4 & 5 & 6 & 7 & 8 & 9 & 10\\ \hline
$q_t$ & $q_{172}$ & $q_{164}$ & $q_{176}$ & $q_{171}$ & $q_{178}$ & $q_{170}$ & $q_{162}$ & $q_{158}$ & $q_{175}$ \\ \hline
$\varepsilon\ge$ &
$1.63\times10^{-5}$ & $3.75\times10^{-6}$ & $3.05\times10^{-7}$ &
$5.6\times10^{-6}$ & $8.72\times10^{-7}$ & $3.41\times10^{-6}$&
$1.41\times10^{-6}$ & $9.71\times10^{-7}$ & $6.85\times10^{-7}$ \\ \hline
$\ell_{1}-1\le$ & 313 & 198 & 159 & 135 & 122 & 112 & 105 & 100 & 96 \\ \hline
\end{tabular}}
\end{table}
Hence,
\[
1\le \ell_{1}\le
\frac{\log\left(7.49q_{172}/(1.63\times10^{-5})\right)}{\log 2}+1
\le 314 .
\]

\begin{step}
\label{step4.4}
We consider Case~1A and prove that, under the assumptions
$n_1-n_2\le 438$, $n_1-n_3\le 456$, and $\ell_{1}\le 314$,
one has $\ell_{2}\le 320$.
\end{step}

We assume $\ell_{2}\ge 24$.
Let
\[
\Gamma_4=
\log\left(
\frac{\sqrt{5}\left(d_1b^{\ell_{1}}-(d_1-d_2)\right)}
{(b-1)\left(1+\alpha^{-(n_1-n_2)}+\alpha^{-(n_1-n_3)}\right)}
\right)
-n_1\log\alpha
+(\ell_{2}+\ell_{3})\log b .
\]
From inequality~(\ref{eq15}) we get $|\Gamma_4|<6.26b^{-\ell_{2}}$ and hence
\[
0<
\left|
(\ell_{2}+\ell_{3})
\left(\frac{\log b}{\log\alpha}\right)
-n_1
+\frac{\log\left(
\frac{\sqrt{5}(d_1b^{\ell_{1}}-(d_1-d_2))}
{(b-1)(1+\alpha^{-(n_1-n_2)}+\alpha^{-(n_1-n_3)})}
\right)}{\log\alpha}
\right|
<13.1b^{-(\ell_{2}-1)} .
\]

Applying Lemma~\ref{lem3} with
\[
\tau:=\frac{\log b}{\log\alpha},\quad
(A,B)=(13.1,b),\quad
w:=\ell_{2}-1,
\]
we obtain:

\begin{table}[H]
\centering
\caption{}
\tiny
\renewcommand{\arraystretch}{1.3}
\resizebox{\textwidth}{!}{
\begin{tabular}{|c|c|c|c|c|c|c|c|c|c|}
\hline
$b$ & 2&3 & 4 & 5 & 6 & 7 & 8 & 9 & 10 \\ \hline
$q_t$ & $q_{172}$ & $q_{164}$ & $q_{176}$ & $q_{171}$ & $q_{178}$ & $q_{170}$ & $q_{162}$ & $q_{158}$ & $q_{175}$ \\ \hline
$\varepsilon\ge$ &
$2.62\times10^{-7}$& $1.36\times10^{-8}$ & $1.87\times10^{-7}$ &
$9.59\times10^{-8}$ & $5.14\times10^{-8}$&
$2.19\times10^{-8}$ &
$1.82\times10^{-8}$& $1.46\times10^{-8}$ & $5.14\times10^{-8}$\\ \hline
$\ell_{2}-1\le$ &319& 204 & 160 & 137 & 124 & 115 & 108 & 102 & 97 \\ \hline
\end{tabular}}
\end{table}

Thus
\[
1\le \ell_{2}\le
\frac{\log\left(13.1q_{172}/(2.62\times10^{-7})\right)}{\log 2}+1
\le 320 .
\]

\begin{step}
\label{step4.5}
We consider Case~1B and prove that, under the assumptions
$n_1-n_2\le 438$ and $\ell_{1}\le 308$, one has
$\ell_{2}\le 316$ or $n_1-n_3\le 455$.
\end{step}
Assume $\ell_{2}, n_1-n_3\ge 24$ and set
\[
\Gamma_5=
\log\left(
\frac{\sqrt{5}(d_1b^{\ell_{1}}-(d_1-d_2))}
{(b-1)(1+\alpha^{n_2-n_1})}
\right)
-n_1\log\alpha
+(\ell_{2}+\ell_{3})\log b .
\]
From inequality~\eqref{eq16} we obtain
$|\Gamma_5|<8.84\max\{b^{-(\ell_{2}-1)},\alpha^{-(n_1-n_3)}\}$ and hence
\[
0<
\left|
(\ell_{2}+\ell_{3})
\left(\frac{\log b}{\log\alpha}\right)
-n_1
+\frac{\log\left(
\frac{\sqrt{5}(d_1b^{\ell_{1}}-(d_1-d_2))}
{(b-1)(1+\alpha^{n_2-n_1})}
\right)}{\log\alpha}
\right|
<
\max\{18.38b^{-(\ell_{2}-1)},18.38\alpha^{-(n_1-n_3)}\}.
\]
\begin{table}[H]
\centering
\caption{Case $(A,B)=(18.38,b)$}
\tiny
\renewcommand{\arraystretch}{1.3}
\resizebox{\textwidth}{!}{
\begin{tabular}{|c|c|c|c|c|c|c|c|c|c|}
\hline
$b$ & 2 & 3 & 4 & 5 & 6 & 7 & 8 & 9 & 10 \\ \hline
$q_t$ & $q_{172}$ & $q_{164}$ & $q_{176}$ & $q_{171}$ & $q_{178}$ & $q_{170}$ & $q_{162}$ & $q_{158}$ & $q_{175}$ \\ \hline
$\varepsilon\ge$ &
$9.76\times10^{-6}$ & $4.41\times10^{-5}$ & $7.06\times10^{-6}$ &
$1.09\times10^{-5}$ & $5.98\times10^{-7}$ &
$1.13\times10^{-5}$ & $3.86\times10^{-6}$ &
$2.07\times10^{-6}$ & $1.64\times10^{-6}$ \\ \hline
$\ell_{2}-1\le$ &315 & 197 & 157 & 135 & 122 & 112 & 105 & 100 & 96 \\ \hline
\end{tabular}}
\end{table}

\begin{table}[H]
\centering
\caption{Case $(A,B)=(18.38,\alpha)$}
\tiny
\renewcommand{\arraystretch}{1.3}
\resizebox{\textwidth}{!}{
\begin{tabular}{|c|c|c|c|c|c|c|c|c|c|}
\hline
$b$ & 2 & 3 & 4 & 5 & 6 & 7 & 8 & 9 & 10 \\ \hline
$q_t$ & $q_{172}$ & $q_{164}$ & $q_{176}$ & $q_{171}$ & $q_{178}$ & $q_{170}$ & $q_{162}$ & $q_{158}$ & $q_{175}$ \\ \hline
$\varepsilon\ge$ &
$9.76\times10^{-6}$ & $4.41\times10^{-5}$ & $7.06\times10^{-6}$ &
$5.28\times10^{-6}$ & $5.98\times10^{-7}$ &
$7.82\times10^{-6}$ & $3.86\times10^{-6}$ &
$2.07\times10^{-6}$ & $1.64\times10^{-6}$ \\ \hline
$n_1-n_3\le$ &453 & 449 & 452 & 451 & 455 & 451 & 452 & 453 & 455 \\ \hline
\end{tabular}}
\end{table}

Therefore,
\[
1\le \ell_{2}\le
\frac{\log\left(18.38q_{172}/(9.76\times10^{-6})\right)}{\log 2}+1
\le 316
\]
or
\[
1\le n_1-n_3\le
\frac{\log\left(18.38q_{178}/(5.98\times10^{-7})\right)}{\log\alpha}
\le 455 .
\]
Within Case~1B we distinguish:\\
\textbf{Case 1B-A:} 
\[
n_1-n_3\le 455;
\]
\textbf{Case 1B-B:} 
\[
\ell_{2}\le 316.
\]
\begin{step}
\label{step4.6}
We consider Case 1B-A and show that under the assumption that 
$n_1-n_2\leq438$, $\ell_{1}\le 308$ and $n_1-n_3\le 455$ 
we have $\ell_{2}\le 319$.
\end{step}
In this step we consider inequality \eqref{eq17} and assume that $\ell_{2}\ge 24$. We set
\[
\Gamma_6:= \log\!\left(
\frac{(b-1)\left(1+\alpha^{-(n_1-n_2)}+\alpha^{-(n_1-n_3)} \right)}
{\sqrt{5}\left(d_1b^{\ell_{1}}-(d_1-d_2) \right)}
\right)
+n_1\log\alpha+(\ell_{2}+\ell_{3})\log b
\]
and inequality \eqref{eq17} yields that $|\Gamma_6|< 3.2b^{-(\ell_{2}-1)}$. Then we get
\[
0<\left|(\ell_{2}+\ell_{3})\left(\frac{\log b}{\log\alpha} \right)-n_1
+\frac{\log\left(\frac{\sqrt{5}\left(d_1b^{\ell_{1}}-(d_1-d_2) \right)}
{(b-1)\left(1+\alpha^{-(n_1-n_2)}+\alpha^{-(n_1-n_3)} \right)} \right)}
{\log\alpha}\right|
< 6.65 b^{-(\ell_{2}-1)} .
\]

Now we have to apply Lemma \ref{lem3} to inequality above by taking the following parameters
\[
\tau:= \frac{\log b}{\log\alpha}, \quad
\mu:= \frac{\log\left(
\frac{\sqrt{5}\left(d_1b^{\ell_{1}}-(d_1-d_2) \right)}
{(b-1)\left(1+\alpha^{-k}+\alpha^{-m} \right)}
\right)}{\log\alpha},
\quad (A, B)=(6.65, b),\quad w:= \ell_{2}-1,
\]
with $d_1, d_2 \in\{0,1,\dots, (b-1)\}$, $d_1\ge1$, for each possible value of 
$n_1-n_2=k=0,1,2\cdots,438$, 
$n_1-n_3=m=0,1,\dots,455$ and 
$\ell_{1}=1,2,\dots,314$ 
(with respect to the obvious condition that $n_1-n_2\le n_1-n_3$).  
We take $M:= 3.1\times10^{86}$.  
With Mathematica we got the following results

\begin{table}[H]
\centering
\caption{}
\resizebox{\textwidth}{!}{
\tiny
\renewcommand{\arraystretch}{1.3}
\begin{tabular}{|c|c|c|c|c|c|c|c|c|c|}
\hline
$b$ & 2 & 3 & 4 & 5 & 6 & 7 & 8 & 9 & 10 \\ \hline
$q_t$ & $q_{172}$ & $q_{164}$ & $q_{176}$ & $q_{171}$ & $q_{178}$ & $q_{170}$ & $q_{162}$ & $q_{158}$ & $q_{175}$ \\ \hline
$\varepsilon\ge$ &
$2.62\times10^{-7}$ & $1.36\times10^{-8}$ & $8.32\times10^{-8}$ &
$9.59\times10^{-8}$ & $1.89\times10^{-8}$ & $6.98\times10^{-8}$ &
$1.82\times10^{-8}$ & $3.08\times10^{-8}$ & $5.14\times10^{-8}$ \\ \hline
$\ell_{2}-1\le$ & 318 & 203 & 160 & 137 & 124 & 114 & 107 & 101 & 97 \\ \hline
\end{tabular}}
\end{table}

In all cases we can conclude that
\[
1\le \ell_{2} \le 
\frac{\log\left(6.65q_{172}/(2.62\times 10^{-7}) \right)}{\log 2}+1
\le 319.
\]

\begin{step}
\label{step4.7}
We consider Case 1B-B and show that under the assumption that 
$n_1-n_2\leq438$, $\ell_{1}\le 308$ and $\ell_{2}\le 319$ 
we have $n_1-n_3\le 463$.
\end{step}

In this step we consider inequality (\ref{eq18}) and assume that $n_1-n_3\ge 24$.\\ 

We set
\[
\Gamma_7:= \log\!\left(
\frac{\sqrt{5}\left(d_1b^{\ell_{1}+\ell_{2}}-(d_1-d_2)b^{\ell_{2}}-(d_1-d_3) \right)}
{(b-1)\left( 1+\alpha^{n_2-n_1}\right)}
\right)
-n_1\log\alpha+\ell_{3}\log b
\]
and inequality (\ref{eq18}) yields that 
$|\Gamma_7|< 2.8\alpha^{-(n_1-n_3)}$. Then we get
\[
0<\left|\ell_{3}\left(\frac{\log b}{\log\alpha} \right)-n_1+
\frac{\log\left(
\frac{\sqrt{5}\left(d_1b^{\ell_{1}+\ell_{2}}-(d_1-d_2)b^{\ell_{2}}-(d_2-d_3) \right)}
{(b-1)\left( 1+\alpha^{-(n_1-n_2)}\right)}
\right)}{\log\alpha}\right|
< 5.82 \alpha^{-(n_1-n_3)} .
\]

Now we have to apply Lemma \ref{lem3} to inequality above by taking the following parameters
\[
\tau:= \frac{\log b}{\log\alpha},\quad
\mu:= \frac{\log\left(
\frac{\sqrt{5}\left(d_1b^{\ell_{1}+\ell_{2}}-(d_1-d_2)b^{\ell_{2}}-(d_1-d_3) \right)}
{(b-1)\left( 1+\alpha^{-k}\right)}
\right)}{\log\alpha},
\quad (A,B)=(5.82,\alpha),\quad w:= n_1-n_3,
\]
with $d_1,d_2,d_3\in\{0,1,\dots,(b-1)\}$, $d_1\ge1$, for each possible value of 
$n_1-n_2=k=0,1,2\cdots,438$, 
$\ell_{1}=1,2,\dots,308$, and 
$\ell_{2}=1,2,\dots,319$ 
(with respect to the obvious condition that $\ell_{1}\le \ell_{2}$).  
We take $M:= 3.1\times10^{86}$.  
With Mathematica we got the following results

\begin{table}[H]
\centering
\caption{}
\resizebox{\textwidth}{!}{
\tiny
\renewcommand{\arraystretch}{1.3}
\begin{tabular}{|c|c|c|c|c|c|c|c|c|c|}
\hline
$b$ & 2 & 3 & 4 & 5 & 6 & 7 & 8 & 9 & 10 \\ \hline
$q_t$ & $q_{172}$ & $q_{164}$ & $q_{176}$ & $q_{171}$ & $q_{178}$ & $q_{170}$ & $q_{162}$ & $q_{158}$ & $q_{175}$ \\ \hline
$\varepsilon\ge$ &
$7.92\times10^{-7}$ & $2.94\times10^{-7}$ & $2.7\times10^{-8}$ &
$7.7\times10^{-8}$ & $2.1\times10^{-8}$ & $1.96\times10^{-7}$ &
$1.23\times10^{-8}$ & $1.91\times10^{-7}$ & $1.5\times10^{-8}$ \\ \hline
$n_1-n_3\le$ & 456 & 457 & 461 & 457 & 459 & 456 & 462 & 456 & 463 \\ \hline
\end{tabular}}
\end{table}

In all cases we can conclude that
\[
1\le n_1-n_3 \le 
\frac{\log\left(5.82q_{176}/(1.5\times10^{-8})\right)}{\log\alpha}
\le 463.
\]

\begin{step}
\label{step4.8}
We consider Case 2 and show that under the assumption that 
$\ell_{1}\leq301$, we have $\ell_{2}\le 310$ or $n_1-n_2\le 451$.
\end{step}

In this step we consider inequality \eqref{eq19} and assume that 
$\ell_{2}, n_1-n_2\ge 24$.
We set
\[
\Gamma'_1:= \log\left(
\frac{\sqrt{5}\left( d_1b^{\ell_{1}}-(d_1-d_2)\right)}{b-1}
\right)
-n_1\log\alpha+(\ell_{2}+\ell_{3})\log b
\]
and inequality \eqref{eq19} yields that 
$|\Gamma'_1|<11.4\max\left\{ b^{-(\ell_{2}-1)},\alpha^{-(n_1-n_2)}\right\}$. 
Then we get
\[
0<\left|(\ell_{2}+\ell_{3})\left(\frac{\log b}{\log\alpha} \right)-n_1+
\frac{\log\left(
\frac{\sqrt{5}\left( d_1b^{\ell_{1}}-(d_1-d_2)\right)}{b-1}
\right)}{\log\alpha}\right|
< \max\left\{23.7b^{-(\ell_{2}-1)},\,23.7\alpha^{-(n_1-n_2)}\right\}
\alpha^{-(n_1-n_3)} .
\]

Now we have to apply Lemma \ref{lem3} to inequality above by taking the following parameters
\[
\tau:= \frac{\log b}{\log\alpha},\quad
\mu:= \frac{\log\left(
\frac{\sqrt{5}\left( d_1b^{\ell_{1}}-(d_1-d_2)\right)}{b-1}
\right)}{\log\alpha},
\]
\[
(A,B)=(23.7,b)\ \text{or}\ (23.7,\alpha),\quad 
w:= \ell_{2}-1 \ \text{or}\ w:= n_1-n_2,
\]
with $d_1,d_2\in\{0,1,\dots,(b-1)\}$, $d_1\ge1$, for each possible value of 
$\ell_{1}=1,2,\dots,301$.  
We take $M:= 3.1\times10^{86}$.  
With Mathematica we got the following results

\begin{table}[H]
\centering
\caption{Case $(A,B)=(23.7,b)$}
\resizebox{\textwidth}{!}{
\tiny
\renewcommand{\arraystretch}{1.3}
\begin{tabular}{|c|c|c|c|c|c|c|c|c|c|}
\hline
$b$ & 2 & 3 & 4 & 5 & 6 & 7 & 8 & 9 & 10 \\ \hline
$q_t$ & $q_{172}$ & $q_{164}$ & $q_{176}$ & $q_{171}$ & $q_{178}$ & $q_{170}$ & $q_{162}$ & $q_{158}$ & $q_{175}$ \\ \hline
$\varepsilon\ge$ &
$6.61\times10^{-4}$ & $2.1\times10^{-4}$ & $4.02\times10^{-4}$ &
$4.45\times10^{-4}$ & $3.26\times10^{-5}$ & $7.01\times10^{-4}$ &
$6.24\times10^{-4}$ & $7.11\times10^{-5}$ & $2.01\times10^{-5}$ \\ \hline
$\ell_{2}-1\le$ & 309 & 196 & 155 & 133 & 120 & 110 & 103 & 98 & 95 \\ \hline
\end{tabular}}
\end{table}

\begin{table}[H]
\centering
\caption{Case $(A,B)=(23.7,\alpha)$}
\resizebox{\textwidth}{!}{
\tiny
\renewcommand{\arraystretch}{1.3}
\begin{tabular}{|c|c|c|c|c|c|c|c|c|c|}
\hline
$b$ & 2 & 3 & 4 & 5 & 6 & 7 & 8 & 9 & 10 \\ \hline
$q_t$ & $q_{172}$ & $q_{164}$ & $q_{176}$ & $q_{171}$ & $q_{178}$ & $q_{170}$ & $q_{162}$ & $q_{158}$ & $q_{175}$ \\ \hline
$\varepsilon\ge$ &
$6.61\times10^{-4}$ & $2.01\times10^{-4}$ & $4.02\times10^{-4}$ &
$4.66\times10^{-4}$ & $3.26\times10^{-4}$ & $2.33\times10^{-5}$ &
$2.03\times10^{-4}$ & $7.11\times10^{-5}$ & $2.01\times10^{-5}$ \\ \hline
$n_1-n_2\le$ & 445 & 446 & 444 & 442 & 447 & 444 & 445 & 446 & 451 \\ \hline
\end{tabular}}
\end{table}

In all cases we can conclude that
\[
1\le \ell_{2} \le 
\frac{\log\left(23.7q_{172}/(6.61\times10^{-4})\right)}{\log 2}+1
\le 310
\]
or
\[
1\le n_1-n_2 \le 
\frac{\log\left(23.7q_{175}/(2.01\times10^{-5})\right)}{\log\alpha}
\le 451.
\]
Within Case~2, we distinguish two further subcases:\\
\textbf{Case 2A:} 
\[
\ell_{2}\le 310,
\]
 \textbf{Case 2B:} 
 \[
 n_1-n_2\le 451.
\]

\begin{step}
\label{step4.9}
We consider Case~2A and show that, under the assumptions that $\ell_{1}\leq301$ and $\ell_{2}\le 310$, we have $n_1-n_2\le 459$.
\end{step}

In this step we consider inequality~(\ref{eq20}) and assume that $n_1-n_2\ge 24$.\\ 

We set
\[
\Gamma'_2:= \log\!\left(\frac{\sqrt{5}\left(d_1b^{\ell_{1}+\ell_{2}}-(d_1-d_2)b^{\ell_{2}}-(d_1-d_3) \right) }{b-1} \right)
-n_1\log\alpha+\ell_{3}\log b ,
\]
and inequality~(\ref{eq20}) yields that $|\Gamma'_2|< 5.54\alpha^{-(n_1-n_2)}$. Then we get
\[
0<\left|\ell_{3}\left(\frac{\log b}{\log\alpha} \right)-n_1+
\frac{\log\left(\frac{\sqrt{5}\left(d_1b^{\ell_{1}+\ell_{2}}-(d_1-d_2)b^{\ell_{2}}-(d_1-d_3) \right) }{b-1} \right)}{\log\alpha}  \right|
< 11.52 \alpha^{-(n_1-n_2)}.
\]

Now we apply Lemma~\ref{lem3} to the above inequality by taking the following parameters
\[
\tau:= \frac{\log b}{\log\alpha}, \qquad
\mu:= \frac{\log\left(\frac{\sqrt{5}\left(d_1b^{\ell_{1}+\ell_{2}}-(d_1-d_2)b^{\ell_{2}}-(d_2-d_3) \right) }{b-1} \right)}{\log\alpha},
\qquad (A,B)=(11.52,\alpha), \qquad w:= n_1-n_2,
\]
with $d_1,d_2,d_3\in\{0,1,\dots,b-1\}$, $d_1\ge1$, for each possible value of $\ell_{1}=1,2,\dots,301$ and $\ell_{2}=1,2,\dots,310$ (with respect to the obvious condition that $\ell_{1}\le \ell_{2}$).  
We take $M:=3.1\times10^{86}$. Using \textit{Mathematica}, we obtained the following results:

\begin{table}[H]
\centering
\caption{}
\renewcommand{\arraystretch}{1.3}
\tiny
\begin{tabular}{|c|c|c|c|c|c|c|c|c|c|}
\hline
$b$ & 2 & 3 & 4 & 5 & 6 & 7 & 8 & 9 & 10 \\ \hline
$q_t$ & $q_{172}$ & $q_{164}$ & $q_{176}$ & $q_{171}$ & $q_{178}$ & $q_{170}$ & $q_{162}$ & $q_{158}$ & $q_{175}$ \\ \hline
$\varepsilon\ge$ & $5.44\cdot10^{-6}$ & $1.71\cdot10^{-6}$ & $7.73\cdot10^{-7}$ & $1.28\cdot10^{-6}$ & $3.92\cdot10^{-7}$ & $3.72\cdot10^{-8}$ & $1.23\cdot10^{-7}$ & $1.29\cdot10^{-6}$ & $1.59\cdot10^{-7}$ \\ \hline
$n_1-n_2\le$ & 453 & 454 & 456 & 453 & 455 & 461 & 459 & 453 & 459 \\ \hline
\end{tabular}
\end{table}

In all cases we conclude that
\[
1\le n_1-n_2 \le 
\frac{\log\left(11.52q_{175}/(1.59\times10^{-7}) \right)}{\log \alpha}
\le 459.
\]

\begin{step}
\label{step4.10}
We consider Case~2A and show that, under the assumptions $\ell_{1}\le301$, $\ell_{2}\le310$ and $n_1-n_2\le459$, we have $n_1-n_3\le463$.
\end{step}
In this step we consider inequality~(\ref{eq21}) and assume that $n_1-n_3\ge24$. We set
\[
\Gamma'_3:= \log\!\left(\frac{\sqrt{5}\left(d_1b^{\ell_{1}+\ell_{2}}-(d_1-d_2)b^{\ell_{2}}-(d_2-d_3) \right) }{(b-1)\left(1+\alpha^{n_2-n_1}\right)} \right)
-n_1\log\alpha+\ell_{3}\log b ,
\]
and inequality~\eqref{eq21} yields that $|\Gamma'_3|<2.8\alpha^{-(n_1-n_3)}$. Then we get
\[
0<\left|\ell_{3}\left(\frac{\log b}{\log\alpha} \right)-n_1+
\frac{\log\left(\frac{\sqrt{5}\left(d_1b^{\ell_{1}+\ell_{2}}-(d_1-d_2)b^{\ell_{2}}-(d_2-d_3) \right) }{(b-1)\left(1+\alpha^{-(n_1-n_2)}\right)} \right)}{\log\alpha}  \right|
< 5.82 \alpha^{-(n_1-n_3)}.
\]
Now we apply Lemma~\ref{lem3} to the above inequality by taking the following parameters
\[
\tau:= \frac{\log b}{\log\alpha}, \qquad
\mu:= \frac{\log\left(\frac{\sqrt{5}\left(d_1b^{\ell_{1}+\ell_{2}}-(d_1-d_2)b^{\ell_{2}}-(d_2-d_3) \right) }{(b-1)\left(1+\alpha^{-k}\right)} \right)}{\log\alpha},
\qquad (A,B)=(5.82,\alpha), \qquad w:= n_1-n_3,
\]
with $d_1,d_2,d_3\in\{0,1,\dots,b-1\}$, $d_1\ge1$, for each possible value of $n_1-n_2=k=0,1,\dots,459$, $\ell_{1}=1,2,\dots,301$ and $\ell_{2}=1,2,\dots,310$ (with respect to the obvious condition that $\ell_{1}\le\ell_{2}$).  
We take $M:=3.1\times10^{86}$. Using \textit{Mathematica}, we obtained the same results as in Step~\ref{step4.7}. Thus we have $1\le n_1-n_3\le463$.

\begin{step}
\label{step4.11}
We consider Case~2B and show that, under the assumptions $\ell_{1}\le301$ and $n_1-n_2\le451$, we have $n_1-n_3\le455$ or $\ell_{2}\le316$.
\end{step}

In this step we use the same parameters as in Step~\ref{step4.5}, except that $\ell_{1}\le301$ and $n_1-n_2\le451$. We obtain the same results as in Step~\ref{step4.5}. Within Case~2B, we distinguish two further subcases:\\
     \textbf{Case 2B-A:} 
     \[
     n_1-n_3\le455, 
     \]
     \textbf{Case 2B-B:} 
\[     
     \ell_{2}\le316.
\]

\begin{step}
\label{step4.12} 
We consider Case~2B-A and show that, under the assumptions $\ell_{1}\le301$, $n_1-n_2\le451$ and $n_1-n_3\le455$, we have $\ell_{2}\le316$.
\end{step}

In this step we use the same parameters as in Step~\ref{step4.6}, except for the updated bounds $\ell_{1}\le301$, $n_1-n_2\le451$ and $n_1-n_3\le455$. The results are identical to those of Step~\ref{step4.6}.

\begin{step}
\label{step4.13} 
We consider Case~2B-B and show that, under the assumptions $\ell_{1}\le301$, $n_1-n_2\le451$, and $\ell_{2}\le316$, we have $n_1-n_3\le463$.
\end{step}

In this step we use the same parameters as in Step~\ref{step4.7}, except for the updated bounds $\ell_{1}\le301$, $n_1-n_2\le451$, and $\ell_{2}\le316$. The results are identical to those of Step~\ref{step4.7}.

Table~\ref{tablex} summarizes the upper bounds obtained so far.

\begin{table}[H]
\centering
\footnotesize
\renewcommand{\arraystretch}{1.3}
\begin{tabular}{|c|c|c|c|}
\hline
\textbf{Upper bound ($\le$)} & \textbf{Case 1A} & \textbf{Case 1B-A} & \textbf{Case 1B-B} \\ \hline
$n_1-n_2$ & 438 & 438 & 438 \\ \hline
$n_1-n_3$ & 456 & 455 & 463 \\ \hline
$\ell_{1}$ & 314 & 308 & 308 \\ \hline
$\ell_{2}$ & 320 & 319 & 316 \\ \hline
\end{tabular}

\vspace{0.5em}
\begin{tabular}{|c|c|c|c|c|}
\hline
\textbf{Upper bound ($\le$)} & \textbf{Case 2A} & \textbf{Case 2B-A} & \textbf{Case 2B-B} & \textbf{Overall} \\ \hline
$n_1-n_2$ & 541 & 451 & 451 & 451 \\ \hline
$n_1-n_3$ & 463 & 455 & 463 & 463 \\ \hline
$\ell_{1}$ & 301 & 301 & 301 & 314 \\ \hline
$\ell_{2}$ & 310 & 316 & 316 & 320 \\ \hline
\end{tabular}
\caption{Summary of upper bounds for all cases}
\label{tablex}
\end{table}

\begin{step}
\label{step4.14}
Under the assumptions $n_1-n_2\le451$, $n_1-n_3\le463$, $\ell_1\le314$, and $\ell_{2}\le320$, we obtain an upper bound for $n_1$.
\end{step}

For the final step, we consider inequality~(\ref{eq22}) and set
\[
\Gamma:= \log\!\left(\frac{\sqrt{5}\left(d_1b^{\ell_{1}+\ell_{2}}-(d_1-d_2)b^{\ell_{2}}-(d_2-d_3) \right) }{(b-1)\left( 1+\alpha^{n_2-n_1}+\alpha^{n_3-n_1}\right)} \right)
-n_1\log\alpha+\ell_{3}\log b ,
\]
and inequality~(\ref{eq22}) yields $|\Gamma|<4.48\alpha^{-n_1}$. Then we obtain
\[
0<\left|\ell_{3}\left(\frac{\log b}{\log\alpha} \right)-n_1+
\frac{\log\left(\frac{\sqrt{5}\left(d_1b^{\ell_{1}+\ell_{2}}-(d_1-d_2)b^{\ell_{2}}-(d_2-d_3) \right) }{(b-1)\left( 1+\alpha^{-(n_1-n_2)}+\alpha^{-(n_1-n_3)}\right)} \right)}{\log\alpha} \right| 
< 9.31 \alpha^{-n_1}.
\]

We then apply Lemma~\ref{lem3} with parameters
\[
\tau:=\frac{\log b}{\log\alpha}, \quad
\mu:= \frac{\log\left(\frac{\sqrt{5}\left(d_1b^{\ell_{1}+\ell_{2}}-(d_1-d_2)b^{\ell_{2}}-(d_2-d_3)\right)}{(b-1)\left(1+\alpha^{-k}+\alpha^{-m}\right)}\right)}{\log\alpha}, \quad
(A,B)=(9.31,\alpha), \quad w:=n_1,
\]
with $d_1,d_2,d_3\in\{0,1,\dots,b-1\}$, $d_1\ge1$, for each $k=0,1,\dots,451$ and $m=0,1,\dots,463$, $\ell_{1}=1,\dots,314$, $\ell_{2}=1,\dots,320$ (with $\ell_{1}\le\ell_{2}$ and $n_1-n_2\le n_1-n_3$).  
We take $M:=3.1\times10^{86}$. We obtain the following results:
\begin{table}[H]
	\centering
	\caption{}
	
	\renewcommand{\arraystretch}{1.3}
	\begin{tabular}{|c|c|c|c|c|c|c|c|c|c|}
		\hline
		$b$ & 2 & 3 & 4 & 5 & 6 & 7 & 8 & 9 & 10 \\ \hline
		$n_1\le$ & 74 & 74 & 71 & 71 & 71 & 69 & 72 & 70 & 69 \\ \hline
	\end{tabular}
\end{table}
In all cases we conclude that $n_1\le 74$.

In the light of the above results, we need to check equation (\ref{eq1}) in the cases $2\le b\le 10$ for $d_1,d_2,d_3\in\{0,1,\dots,b-1\}$, $d_1\ge1$, $1\le n_1\le 74$ and $\ell_{1}+\ell_{2}+\ell_{3}\le 76$. A quick inspection using Mathematica reveals that Diophantine equation (\ref{eq1}) has only solution listed in the statement of Theorem \ref{thrm2}.
\hfill $ \square $
\begin{remark}
	For this last step to reduce the upper bound on $n_1$, the numerical resolution of the problem encountered a major technical difficulty due to its combinatorial structure, which involves seven nested loops, leading to an exponential growth in the number of iterations as the base $b$ increases. A theoretical complexity analysis showed that, for values such as $b=10$, a classical sequential execution in Wolfram Mathematica would have required an estimated computation time of nearly 248  years, making any exhaustive exploration unrealistic on a standard architecture.
	To overcome this limitation, the algorithm was ported to a parallel computing architecture using a NVIDIA RTX A2000 GPU through the CUDA framework. This approach enables the simultaneous distribution of independent iterations across thousands of processing cores, thereby transforming the problem into a massively parallelizable task. The observed performance gain exceeds a $20,000\times$ speed-up compared to the sequential execution: for instance, the computation for $b=2$, initially estimated at 13 months in sequential mode, was completed in approximately 8 minutes, while the case $b=3$, estimated at 5 years, was completed in about 50 minutes. Thanks to this architecture, higher bases (up to $b=10$) were processed within only a few hours, making it possible to effectively determine the numerical bounds presented in this work.
\end{remark}
\section*{Acknowledgements}

The authors sincerely express their gratitude to the Ministère de l’Efficacité du Service Public et de la Transformation Numérique, through Togo Data Lab, for their invaluable technical support and for granting access to their high-performance computing infrastructure. The large-scale numerical computations required in this work were made possible thanks to their advanced parallel computing environment, including GPU-based architecture, which enabled the efficient execution of highly demanding combinatorial calculations.

The authors particularly acknowledge the expertise of the Togo Data Lab team in optimizing and deploying the algorithm on a parallel architecture, thereby significantly reducing computation times and ensuring the reliability of the numerical results presented in this article.

\section*{Conflict of interest}
On behalf of all authors, the corresponding author states that there is no conflict of interest.

\end{document}